\documentclass[a4paper,11pt]{amsart}
\usepackage{amsfonts,amssymb,amsthm,cite,amsmath,amstext}
\usepackage{color}
\usepackage[colorlinks,linkcolor=black,citecolor=black]{hyperref}
\usepackage{comment}
\usepackage{framed}

\definecolor{shadecolor}{gray}{0.875}

\newtheorem{thrm}{Theorem}[section]
\newtheorem{lem}[thrm]{Lemma}
\newtheorem{cor}[thrm]{Corollary}
\newtheorem{prop}[thrm]{Proposition}

\theoremstyle{definition}
\newtheorem{defn}[thrm]{Definition}

\newtheorem{exmple}[thrm]{Example}
\newtheorem{rmk}[thrm]{Remark}
\newtheorem{ques}[thrm]{Question}

\DeclareMathOperator{\Eff}{\overline{Eff}}
\DeclareMathOperator{\Nef}{Nef}
\DeclareMathOperator{\Mov}{Mov}

\DeclareMathOperator{\Supp}{Supp}

\DeclareMathOperator{\vol}{vol}

\DeclareMathOperator{\CI}{CI}

\DeclareMathOperator{\mult}{mult}
\DeclareMathOperator{\Amp}{Amp}
\DeclareMathOperator{\Cvx}{Cvx}

\DeclareMathOperator{\HConc}{HConc}

\DeclareMathOperator{\BC}{BC}

\title{\textbf{Convexity and Zariski decomposition structure}}
\author{{Brian Lehmann and Jian Xiao}}
\date{}
\begin{document}

\begin{abstract}
This is the first part of our work on Zariski decomposition structures, where we study Zariski decompositions using Legendre-Fenchel type transforms.  In this way we define a Zariski decomposition for curve classes.  This decomposition enables us to develop the theory of the volume function for curves defined by the second named author, yielding some fundamental positivity results for curve classes.  For varieties with special structures, the Zariski decomposition for curve classes admits an interesting geometric interpretation.
\end{abstract}

\maketitle


\section{Introduction}
In \cite{zariski62} Zariski introduced a fundamental tool for studying linear series on a surface now known as a Zariski decomposition.  Over the past 50 years the Zariski decomposition and its generalizations to divisors in higher dimensions have played a central role in birational geometry.  In this paper we apply abstract convex analysis to the study of Zariski decompositions.  The key perspective is that a Zariski decomposition captures the failure of strict log concavity of a volume function, and thus can be studied using Legendre-Fenchel type transforms.  Surprisingly,  such transforms capture rich geometric information about the variety, a posteriori motivating many well-known geometric inequalities for pseudo-effective divisors.  

There are two natural dualities for cones of divisors and curves: the nef cone of divisors $\Nef^{1}(X)$ is dual to the pseudo-effective cone of curves $\Eff_{1}(X)$ and the pseudo-effective cone of divisors $\Eff^{1}(X)$ is dual to the movable cone of curves $\Mov_{1}(X)$.  In this paper we study the first duality, obtaining a Zariski decomposition for curve classes on varieties of arbitrary dimension which generalizes Zariski's original construction.  In the sequel \cite{lehmannxiao2015b}, we will focus on the second duality and study $\sigma$-decompositions from the perspective of convex analysis.

Throughout we work over $\mathbb{C}$, but the main results also hold over an algebraically closed field or in the K\"ahler setting (see Section \ref{outlinesec}).

\subsection{Zariski decomposition}
We define a Zariski decomposition for big curve classes -- elements of the interior of the pseudo-effective cone of curves $\Eff_{1}(X)$.

\begin{defn}
\label{def zariski decomposition}
Let $X$ be a projective variety of dimension $n$ and let $\alpha \in \Eff_{1}(X)^{\circ}$ be a big curve class. Then a Zariski decomposition for $\alpha$ is a decomposition
\begin{align*}
\alpha = B^{n-1} + \gamma
\end{align*}
where $B$ is a big and nef $\mathbb{R}$-Cartier divisor class, $\gamma$ is pseudo-effective, and $B \cdot \gamma = 0$.  We call $B^{n-1}$ the ``positive part'' and $\gamma$ the ``negative part" of the decomposition.
\end{defn}

This definition directly generalizes Zariski's original definition, which (for big classes) is given by similar intersection criteria.  As we will see shortly in Section \ref{zardecomandconvexity}, it also mirrors the $\sigma$-decomposition of \cite{Nak04} and the Zariski decomposition of \cite{fl14}.  Our first theorem is:

\begin{thrm}
\label{thm decomposition mainthrm}
Let $X$ be a projective variety of dimension $n$ and let $\alpha \in \Eff_{1}(X)^{\circ}$ be a big curve class. Then $\alpha$ admits a unique Zariski decomposition $\alpha = B^{n-1} + \gamma$.
\end{thrm}


\begin{exmple}
\label{example zariski surface}
If $X$ is an algebraic surface, then the Zariski decomposition provided by Theorem \ref{thm decomposition mainthrm} coincides (for big classes) with the numerical version of the classical definition of \cite{zariski62}.  Indeed, using Proposition \ref{negpartrigid} one sees that the negative part $\gamma$ is represented by an effective curve $N$.  The self-intersection matrix of $N$ must be negative-definite by the Hodge Index Theorem.  (See e.g.~\cite{Nak04} for another perspective focusing on the volume function.)
\end{exmple}

\subsection{Convexity and Zariski decompositions} \label{zardecomandconvexity}
According to the philosophy of \cite{fl14}, the key property of the Zariski decomposition (or $\sigma$-decomposition for divisors) is that it captures the failure of the volume function to be strictly log-concave.  The Zariski decomposition for curves plays a similar role for the following interesting  volume-type function defined in \cite{xiao15}.

\begin{defn} (see \cite[Definition 1.1]{xiao15}) \label{defn:widehatvol}
Let $X$ be a projective variety of dimension $n$ and let $\alpha \in \Eff_{1}(X)$ be a pseudo-effective curve class.  Then the volume of $\alpha$ is defined to be

\begin{align*}
\widehat{\vol}(\alpha) = \inf_{A \textrm{ big and nef divisor class}} \left( \frac{A \cdot \alpha}{\vol(A)^{1/n}} \right)^{\frac{n}{n-1}}.
\end{align*}
We say that a big and nef divisor class $A$ computes $\widehat{\vol}(\alpha)$ if this infimum is achieved by $A$.  When $\alpha$ is a curve class that is not pseudo-effective, we set $\widehat{\vol}(\alpha)=0$.
\end{defn}

The function $\widehat{\vol}$ is a polar transformation of the volume function for ample divisors.  In our setting, the polar transformation plays the role of the Legendre-Fenchel transform of classical convex analysis, linking the differentiability of a function to the strict convexity of its transform.  From this viewpoint, Definition \ref{def zariski decomposition} is important precisely because it captures the log concavity of $\widehat{\vol}$.

\begin{thrm}
\label{thm intro strict concavity M hatvol}
Let $X$ be a smooth projective variety of dimension $n$. Let $\alpha_1, \alpha_2\in \Eff_1 (X)$ be two big curve classes.
Then $$\widehat{\vol}(\alpha_1+ \alpha_2)^{n-1/n}\geq \widehat{\vol}(\alpha_1)^{n-1/n}+ \widehat{\vol}(\alpha_2)^{n-1/n}$$ with equality if and only if the positive parts in the Zariski decompositions of $\alpha_1$ and $\alpha_2$ are proportional.
\end{thrm}

As an important special case, the positive part of a curve class has the same volume as the original class, showing the similarity with the $\sigma$-decomposition.  Furthermore, just in Zariski's classical work, the ``projection'' onto the positive part elucidates the intersection-theoretic nature of the volume.

\begin{thrm}
Let $X$ be a projective variety of dimension $n$ and let $\alpha \in \Eff_{1}(X)^{\circ}$ be a big curve class. Suppose that $\alpha = B^{n-1} + \gamma$ is the Zariski decomposition of $\alpha$.  Then
\begin{equation*}
\widehat{\vol}(\alpha) = \widehat{\vol}(B^{n-1}) = B^{n}
\end{equation*}
and $B$ is the unique big and nef divisor class with this property satisfying $B^{n-1} \preceq \alpha$.
\end{thrm}

\begin{exmple}
An important feature of Zariski decompositions and $\widehat{\vol}$ for curves is that they can be calculated via intersection theory directly on $X$ once one has identified the nef cone of divisors.  (In contrast, the analogous divisor constructions may require passing to birational models of $X$ to admit an interpretation via intersection theory.)  This is illustrated by Example \ref{example proj bundle} where we calculate the Zariski decomposition of any curve class on the projective bundle over $\mathbb{P}^{1}$ defined by $\mathcal{O} \oplus \mathcal{O} \oplus \mathcal{O}(-1)$.
\end{exmple}

\subsection{Formal Zariski decompositions}
The Zariski decomposition for curves can be deduced from a general theory of duality for log concave homogeneous functions defined on cones.  We define a ``formal'' Zariski decomposition capturing the failure of strict log concavity of a certain class of homogeneous functions on finite-dimensional cones.

Let $\mathcal{C}$ be a full dimensional closed proper convex cone in a finite dimensional vector space.  For any $s>1$, let $\HConc_{s}(\mathcal{C})$ denote the collection of functions $f: \mathcal{C} \to \mathbb{R}$ that are upper-semicontinuous, homogeneous of weight $s>1$, strictly positive on the interior of $\mathcal{C}$, and which are $s$-concave in the sense that
\begin{equation*}
f(v)^{1/s} + f(x)^{1/s} \leq f(x+v)^{1/s}
\end{equation*}
for any $v,x \in \mathcal{C}$.  In this context, the correct analogue of the Legendre-Fenchel transform is the (concave homogeneous) polar transform.  For any $f \in \HConc_{s}(\mathcal{C})$, the polar $\mathcal{H}f$ is an element of $\HConc_{s/s-1}(\mathcal{C}^{*})$ for the dual cone $\mathcal{C}^{*}$ defined as
\begin{equation*}
\mathcal{H}f(w^{*}) = \inf_{v \in \mathcal{C}^{\circ}} \left(  \frac{w^{*} \cdot v}{f(v)^{1/s}} \right)^{s/s-1}  \qquad \qquad \forall w^{*} \in \mathcal{C}^{*}.
\end{equation*}
We define what it means for $f \in \HConc_{s}(\mathcal{C})$ to have a Zariski decomposition structure and show that it follows from a differentiability condition for $\mathcal{H}f$, and vice versa (see Section \ref{formal zariski section}).  Just as in the classical definition of Zariski, one can view this structure as a decomposition of the elements of $\mathcal{C}^{\circ}$ into ``positive parts'' retaining the value of $f$ and ``negative parts'' along which the strict log concavity of $f$ fails.

\begin{exmple}
Let $q$ be a bilinear form on a vector space $V$ of signature $(1,\dim V - 1)$ and set $f(v) = q(v,v)$.  Suppose $\mathcal{C}$ is a closed full-dimensional convex cone on which $f$ is non-negative.  Identifying $V$ with $V^*$ under $q$, we see that $\mathcal{C} \subset \mathcal{C}^{*}$ and that $\mathcal{H}f|_{\mathcal{C}} = f$ by the Hodge inequality.  Then $\mathcal{H}f$ on the entire cone $\mathcal{C}^*$ is controlled by a Zariski decomposition with positive parts lying in $\mathcal{C}$.

This is of course the familiar picture for surfaces, where $f$ is the self-intersection on the nef cone and $\mathcal{H}f$ is the volume on the pseudo-effective cone.  Thus we see that the conclusion of Example \ref{example zariski surface} -- that $\vol$ and $\widehat{\vol}$ coincide on surfaces -- is a direct consequence of the Hodge Index Theorem for surfaces.  Furthermore, we obtain a theoretical perspective motivating the linear algebra calculations of \cite{zariski62}.
\end{exmple}

Many of the basic geometric inequalities in algebraic geometry -- and hence for polytopes or convex bodies via toric varieties (as in \cite{teissier82} and \cite{khovanskii88} and the references therein) -- can be understood using this abstract framework.  A posteriori this theory motivates many well-known theorems about the volume of divisors (which can itself be interpreted as a polar transform). In particular, the $\sigma$-decomposition for divisor classes can be also interpreted by our general theory.  See Remark \ref{aposteriori} for more details.

\subsection{Positivity of curves}
The volume function for curves shares many of the important properties of the volume function for divisors.  This is no accident -- as explained above, polar duality behaves compatibly with many topological properties and with geometric inequalities.  Clearly the volume function is homogeneous and it is not hard to show that it is positive precisely on the big cone of curves.  Perhaps the most important property is the following description of the derivative, which mirrors the results of \cite{bfj09} and \cite{lm09} for divisors.

\begin{thrm}
\label{thm C1}
Let $X$ be a projective variety of dimension $n$. Then the function $\widehat{\vol}$ is $\mathcal{C}^{1}$ on the big cone of curves.  More precisely, let $\alpha$ be a big curve class on $X$ and write $\alpha = B^{n-1} + \gamma$ for its Zariski decomposition. For any curve class $\beta$, we have
\begin{align*}
\left. \frac{d}{dt} \right|_{t = 0} \widehat{\vol}(\alpha + t \beta)  = \frac{n}{n-1} B \cdot \beta.
\end{align*}
\end{thrm}

Another key property of the $\sigma$-decomposition for divisors is that the negative part is effective.  While the negative part of the Zariski decomposition for curves need not be effective, the correct analogue is given by the following proposition.

\begin{prop}
Let $X$ be a projective variety of dimension $n$.  Let $\alpha$ be a big curve class and write $\alpha = B^{n-1} + \gamma$ for its Zariski decomposition.  There is a proper subscheme $i: V \subsetneq X$ and a pseudo-effective class $\gamma' \in N_{1}(V)$ such that $i_{*}\gamma' = \gamma$.
\end{prop}

By analogy with the algebraic Morse inequality for nef divisors, we prove a Morse-type inequality for curves.

\begin{thrm}
\label{thm morseinequality}
Let $X$ be a smooth projective variety of dimension $n$.  Let $\alpha$ be a big curve class and let $\beta$ be a nef curve class.  Write $\alpha = B^{n-1} + \gamma$ for the Zariski decomposition of $\alpha$.  If
\begin{equation*}
\widehat{\vol}(\alpha) - n B \cdot \beta > 0
\end{equation*}
then $\alpha - \beta$ is big.
\end{thrm}

\subsection{Examples}
The Zariski decomposition is particularly striking for varieties with a rich geometric structure.  We discuss several examples: toric varieties, Mori dream spaces and hyperk\"ahler manifolds.

The complete intersection cone $\CI_{1}(X)$ is defined to be the closure of the set of classes of the form $A^{n-1}$ for an ample divisor $A$ on $X$.  Note that the positive part of the Zariski decomposition takes values in $\CI_{1}(X)$.  We should emphasize that $\CI_{1}(X)$ need not be convex -- the appendix gives an explicit example.

\subsubsection{Toric varieties}
Let $X$ be a simplicial projective toric variety of dimension $n$ defined by a fan $\Sigma$.  Suppose that the curve class $\alpha$ lies in the interior of the movable cone of curves, or equivalently, $\alpha$ is defined by a positive Minkowski weight on the rays of $\Sigma$.  A classical theorem of Minkowski attaches to such a weight a polytope $P_{\alpha}$ whose facet normals are the rays of $\Sigma$ and whose facet volumes are determined by the weights.

In this setting, the volume is calculated by an mixed volume problem: fixing $P_{\alpha}$, amongst all polytopes whose normal fan refines $\Sigma$ there is a unique $Q$ (up to homothety) minimizing the mixed volume calculation
\begin{equation*}
\left( \frac{V(P_{\alpha}^{n-1},Q)}{\vol(Q)^{1/n}} \right)^{n/n-1}.
\end{equation*}
Then the volume is $n!$ times this minimum value, and the positive part of $\alpha$ is proportional to the $(n-1)$-product of the big and nef divisor defined by $Q$.  Note that if we let $Q$ vary over all polytopes then the Brunn-Minkowski inequality shows that the minimum is given by $Q=cP_{\alpha}$, but the normal fan condition on $Q$ yields a new version of this classical problem.  

We give a procedure for computing the volume of any big curve class $\alpha$ using Zariski decompositions.  From the viewpoint of convex analysis, the compatibility with the Zariski decomposition corresponds to the fact that the solution of an mixed volume problem should be given by a condition on the derivative.

The procedure is as follows.  Note that every big and nef divisor on $X$ is semi-ample (that is, the pullback of an ample divisor on a toric birational model).  Thus, the Zariski decomposition for curves is characterized by the existence of a birational toric morphism $\pi: X \to X'$ such that:
\begin{itemize}
\item the class $\pi_{*}\alpha \in N_{1}(X')$ coincides with $A^{n-1}$ for some ample divisor $A$, and
\item $\alpha - (\pi^*A)^{n-1}$ is pseudo-effective.
\end{itemize}
Thus one can compute the Zariski decomposition and volume for $\alpha$ by the following procedure.
\begin{enumerate}
\item For each toric birational morphism $\pi: X \to X'$, check whether $\pi_{*}\alpha$ is in the complete intersection cone.  If so, there is a unique big and nef divisor $A_{X'}$ such that $A_{X'}^{n-1} = \pi_{*}\alpha$.
\item Check if $\alpha - (\pi^{*}A_{X'})^{n-1}$ is pseudo-effective.
\end{enumerate}
We analyze these in a simple toric variety; see Example \ref{example toric zariski}.

\subsubsection{Hyperk\"ahler manifolds}
For a hyperk\"ahler manifold $X$, the results of \cite[Section 4]{Bou04} show that the volume and $\sigma$-decomposition of divisors satisfy a natural compatibility with the Beauville-Bogomolov form.  We prove the analogous properties for curve classes.  The following theorem is phrased in the K\"ahler setting, although the analogous statements in the projective setting are also true.

\begin{thrm}
Let $X$ be a hyperk\"ahler manifold of dimension $n$ and let $q$ denote the bilinear form on $H^{n-1,n-1}(X)$ induced via duality from the Beauville-Bogomolov form on $H^{1,1}(X)$.
\begin{enumerate}
\item The cone of complete intersection $(n-1,n-1)$-classes is $q$-dual to the cone of pseudo-effective $(n-1,n-1)$-classes.
\item If $\alpha$ is a complete intersection $(n-1,n-1)$-class then $\widehat{\vol}(\alpha) = q(\alpha,\alpha)^{n/2(n-1)}$.
\item Suppose $\alpha$ lies in the interior of the cone of pseudo-effective $(n-1,n-1)$-classes and write $\alpha = B^{n-1} + \gamma$ for its Zariski decomposition.  Then $q(B^{n-1},\gamma) = 0$ and if $\gamma$ is non-zero then $q(\gamma,\gamma)<0$.
\end{enumerate}
\end{thrm}

\subsubsection{Mori dream spaces}
If $X$ is a Mori dream space, then the movable cone of divisors admits a chamber structure defined via the ample cones on small $\mathbb{Q}$-factorial modifications.  This chamber structure behaves compatibly with the $\sigma$-decomposition and the volume function for divisors.

For curves we obtain a complementary picture.  The movable cone of curves admits a ``chamber structure'' defined via the complete intersection cones on small $\mathbb{Q}$-factorial modifications.  However, the Zariski decomposition and volume of curves are no longer invariant under small $\mathbb{Q}$-factorial modifications but instead exactly reflect the changing structure of the pseudo-effective cone of curves.  Thus the Zariski decomposition is the right tool to understand the birational geometry of movable curves on $X$.  This example is analyzed in \cite{lehmannxiao2015b}, since it relies on the techniques developed there.

\subsection{Connections with birational geometry}
Finally, we briefly discuss the relationship between the volume function for curves and several other topics in birational geometry.  A basic technique in birational geometry is to bound the positivity of a divisor using its intersections against specified curves.  These results can profitably be reinterpreted using the volume function of curves.

\begin{prop} \label{multbound}
Let $X$ be a smooth projective variety of dimension $n$.  Choose positive integers $\{k_{i} \}_{i=1}^{r}$.  Suppose that $\alpha \in \Mov_{1}(X)$ is represented by a family of irreducible curves such that for any collection of general points $x_{1},x_{2},\ldots,x_{r},y$ of $X$, there is a curve in our family which contains $y$ and contains each $x_{i}$ with multiplicity $\geq k_{i}$.  Then
\begin{equation*}
\widehat{\vol}(\alpha)^{{n-1}/{n}} \geq \frac{\sum_{i} k_{i}}{r^{1/n}}.
\end{equation*}
\end{prop}

We can thus apply volumes of curves to study Seshadri constants, bounds on volume of divisors, and other related topics.  We defer a more in-depth discussion to Section \ref{applications sec}, contenting ourselves with a fascinating example.

\begin{exmple}
If $X$ is rationally connected, it is interesting to analyze the possible volumes for classes of special rational curves on $X$.  When $X$ is a Fano variety of Picard rank $1$, these invariants will be closely related to classical invariants such as the length and degree.

For example, we say that $\alpha \in N_{1}(X)$ is a rationally connecting class if for any two general points of $X$ there is a chain of rational curves of class $\alpha$ connecting the two points.  Is there a uniform upper bound (depending only on the dimension) for the minimal volume of a rationally connecting class on a rationally connected $X$?  \cite{KMM92} and \cite{campana92} show that this is true for smooth Fano varieties.  We discuss this  question briefly in Section \ref{rcexample}.
\end{exmple}

\subsection{Outline of paper} \label{outlinesec}
In this paper we will work with projective varieties over $\mathbb{C}$ for simplicity of arguments and for compatibility with cited references.  However, all the results will extend to smooth varieties over arbitrary algebraically closed fields on the one hand and arbitrary compact K\"ahler manifolds on the other.  We give a general framework for this extension in Sections \ref{charp background sec} and \ref{kahler background sec} and then explain the details as we go.

In Section \ref{section preliminaries} we review the necessary background, and make several notes explaining how the proofs can be adjusted to arbitrary algebraically closed fields and compact K\"ahler manifolds. Sections \ref{legendresection} and \ref{formal zariski section} discuss polar transforms and formal Zariski decompositions for log concave functions.  In Section \ref{section zariski} we construct the Zariski decomposition of curves and study its basic properties and its relationship with $\widehat{\vol}$.
Section \ref{toric section} discusses toric varieties, and
Section \ref{hyperkahler section} is devoted to the study of hyperk\"ahler manifolds.  Section \ref{applications sec} discusses connections with other areas of birational geometry.  Finally, the appendix collects some ``reverse" Khovanskii-Teissier type results in the analytic setting and a result related to the transcendental holomorphic Morse inequality.  The appendix also gives a toric example where the complete intersection cone of curves is not convex.

\subsection*{Acknowledgements}
We thank M.~Jonsson for his many helpful comments.  Some of the material on toric varieties was worked out in a conversation with J.~Huh, and we are very grateful for his help.  Lehmann would like to thank C.~Araujo, M.~Fulger, D.~Greb, S.~Payne, D.~Treumann, and D.~Yang for helpful conversations.  Xiao would like to thank his supervisor J.-P.~Demailly for suggesting an intersection-theoretic approach to study volume function, and thank W.~Ou for helpful conversations, and thank the China Scholarship Council for the support.

\section{Preliminaries}
\label{section preliminaries}
In this section, we first fix some notations over a projective variety $X$:
\begin{itemize}
\item $N^1(X)$: the real vector space of numerical classes of divisors;
\item $N_1(X)$: the real vector space of numerical classes of curves;
\item $\Eff^1(X)$: the cone of pseudo-effective divisor classes;
\item $\Nef^1(X)$: the cone of nef divisor classes;
\item $\Mov^1(X)$: the cone of movable divisor classes;
\item $\Eff_1(X)$: the cone of pseudo-effective curve classes;
\item $\Mov_1(X)$: the cone of movable curve classes, equivalently by \cite{BDPP13} the dual of $\Eff^{1}(X)$;
\item $\CI_1(X)$: the closure of the set of all curve classes of the form $A^{n-1} $ for an ample divisor $A$.
\end{itemize}
With only a few exceptions, capital letters $A,B,D,L$ will denote $\mathbb{R}$-Cartier divisor classes and greek letters $\alpha,\beta,\gamma$ will denote curve classes.  For two curve classes $\alpha, \beta$, we write $\alpha\succeq \beta$ (resp. $\alpha\preceq \beta$) to denote that $\alpha-\beta$ (resp. $\beta-\alpha$) belongs to $\Eff_{1}(X)$.  We will do similarly for divisor classes, or two elements of a cone $\mathcal{C}$ if the cone is understood.

We will use the notation $\langle - \rangle$ for the positive product on smooth varieties as in \cite{BDPP13}, \cite{bfj09} and \cite{Bou02}. 


To extend our results to arbitrary compact K\"ahler manifolds, we need to deal with transcendental objects which are not given by divisors or curves. Let $X$ be a compact K\"ahler manifold of dimension $n$.  By analogue with the projective situation, we need to deal with the following spaces and positive cones:
\begin{itemize}
\item $H^{1,1}_{\BC}(X, \mathbb{R})$: the real Bott-Chern cohomology group of bidegree $(1,1)$;
\item $H^{n-1,n-1}_{\BC}(X, \mathbb{R})$: the real Bott-Chern cohomology group of bidegree $(n-1,n-1)$;
\item $\mathcal{N}(X)$: the cone of pseudo-effective $(n-1,n-1)$-classes;
\item $\mathcal{M}(X)$: the cone of movable $(n-1,n-1)$-classes;
\item $\overline{\mathcal{K}}(X)$: the cone of nef $(1,1)$-classes, equivalently the closure of the K\"ahler cone;
\item $\mathcal{E}(X)$:  the cone of pseudo-effective $(1,1)$-classes.
\end{itemize}
Recall that we call a Bott-Chern class pseudo-effective if it contains a $d$-closed positive current, and call an $(n-1,n-1)$-class movable if it is contained in the closure of the cone generated by the classes of the form $\mu_{*}(\widetilde{\omega}_1 \wedge...\wedge \widetilde{\omega}_{n-1})$ where $\mu: \widetilde{X}\rightarrow X$ is a modification and $\widetilde{\omega}_1,...,\widetilde{\omega}_{n-1}$ are K\"ahler metrics on $\widetilde{X}$. For the basic theory of positive currents, we refer the reader to \cite{Dem}.

If $X$ is a smooth projective variety over $\mathbb{C}$, then we have the following relations (see e.g. \cite{BDPP13}) $$\Nef^1(X)=\overline{\mathcal{K}}(X)\cap N^1(X),\ \Eff^1(X)=\mathcal{E}(X)\cap N^1(X)$$ and $$\Eff_1(X)=\mathcal{N}(X)\cap N_1(X),\ \Mov_1(X)=\mathcal{M}(X)\cap N_1(X).$$

\subsection{Khovanskii-Teissier inequalities}
We collect several results which we will frequently use in our paper. In every case, the statement for arbitrary projective varieties follows from the familiar smooth versions via a pullback argument.  Recall
the well-known Khovanskii-Teissier inequalities for a pair of nef divisors over projective varieties (see e.g. \cite{Tei79}).
\begin{itemize}
  \item Let $X$ be a projective variety and let $A,B$ be two nef divisor classes on $X$. Then we have
      \begin{equation*}
       A^{n-1} \cdot B \geq (A^{n})^{n-1/n} (B^{n})^{1/n}
      \end{equation*}
\end{itemize}

We also need the characterization of the equality case in the above inequality as in \cite[Theorem D]{bfj09} -- see also \cite{FX14} for the analytic proof for transcendental classes in the K\"ahler setting. (We call this characterization Teissier's proportionality theorem as it was first proposed and studied by B.~Teissier.)
\begin{itemize}
  \item Let $X$ be a projective variety and let $A,B$ be two big and nef divisor classes on $X$.  Then
      \begin{equation*}
      A^{n-1} \cdot B = (A^{n})^{n-1/n} (B^{n})^{1/n}
      \end{equation*}
if and only if $A$ and $B$ are proportional.
\end{itemize}

We next prove a more general version of Teissier's proportionality theorem for $n$ big and nef $(1,1)$-classes over compact K\"ahler manifolds (thus including projective varieties defined over $\mathbb{C}$) which follows easily from the result of \cite{FX14}. This result should be useful in the study of the structure of complete intersection cone $\CI_1 (X)$.

\begin{thrm}
\label{thm KT inequality}
Let $X$ be a compact K\"ahler manifold of dimension $n$, and let $B_1,...,B_n$ be $n$ big and nef $(1,1)$-classes over $X$. Then we have
\begin{align*}
B_1 \cdot B_2 \cdot\cdot\cdot B_n \geq (B_1 ^n)^{1/n} \cdot (B_2 ^n)^{1/n} \cdot\cdot\cdot (B_n ^n)^{1/n},
\end{align*}
where the equality is obtained if and only if $B_1,...,B_n$ are proportional.
\end{thrm}

We include a proof, since we are not aware of any reference in the literature. The proof reduces the global inequalities to the pointwise Brunn-Minkowski inequalities by solving Monge-Amp\`{e}re equations \cite{FX14} (see also \cite{Dem93} for a related result), and then applies the result of \cite{FX14} -- where the key technique and estimates go back to \cite{fuxiao14kcone} -- for a pair of big and nef classes (see also \cite[Theorem D]{bfj09} for divisor classes).

Recall that the ample locus $\Amp(D)$ of a big $(1,1)$-class $D$ is the set of points $x\in X$ such that there is a strictly positive current $T_x \in D$ with analytic singularities which is smooth near $x$.   When $L$ is a big $\mathbb{R}$-divisor class on a smooth projective variety $X$, then the ample locus $\Amp(L)$ is equal to the complement of the augmented base locus $\mathbb{B}_{+}(L)$ (see \cite{Bou04}).

\begin{proof}
Without loss of generality, we can assume all the $B_i ^n=1$. Then we need to prove
\begin{align*}
B_1 \cdot B_2 \cdot\cdot\cdot B_n \geq 1,
\end{align*}
with the equality obtained if and only if $B_1,...,B_n$ are equal.

To this end, we fix a smooth volume form $\Phi$ with $\vol(\Phi) =1$. We choose a smooth $(1,1)$-form $b_j$ in the class $B_j$. Then by \cite[Theorem C]{BEGZ10MAbig}, for every class $B_j$ we can solve the following singular Monge-Amp\`{e}re equation
\begin{align*}
\langle (b_j + i\partial \bar \partial \psi_j)^n \rangle = \Phi,
\end{align*}
where $\langle - \rangle$ denotes the non-pluripolar products of positive currents (see \cite[Definition 1.1 and Proposition 1.6]{BEGZ10MAbig}).

Denote $T_j = b_j + i\partial \bar \partial \psi_j$, then \cite[Theorem B]{BEGZ10MAbig} implies $T_j$ is a positive current with minimal singularities in the class $B_j$. Moreover, $T_j$ is a K\"ahler metric over the ample locus $\Amp(B_j)$ of the big class $B_j$ by \cite[Theorem C]{BEGZ10MAbig}.

Note that $\Amp(B_j)$ is a Zariski open set of $X$. Denote $\Omega=\Amp(B_1)\cap ... \cap \Amp(B_n)$, which is also a Zariski open set. By \cite[Definition 1.17]{BEGZ10MAbig}, we then have
\begin{align*}
B_1 \cdot B_2 \cdot\cdot\cdot B_n &= \int_X \langle T_1 \wedge ... \wedge T_n\rangle\\
&= \int_\Omega T_1 \wedge ... \wedge T_n,
\end{align*}
where the second line follows because the non-pluripolar product $\langle T_1 \wedge ... \wedge T_n\rangle$ puts no mass on the subvariety $X\setminus \Omega$ and all the $T_j$ are K\"ahler metrics over $\Omega$.

For any point $x\in \Omega$, we have the following pointwise Brunn-Minkowski inequality
\begin{align*}
T_1 \wedge ... \wedge T_n \geq \left(\frac{T_1 ^n}{\Phi}\right)^{1/n}\cdot\cdot\cdot \left(\frac{T_n ^n}{\Phi}\right)^{1/n} \Phi= \Phi
\end{align*}
with equality if and only if the K\"ahler metrics $T_j$ are proportional at $x$. Here the second equality follows because we have $T_j ^n =\Phi$ on $\Omega$. In particular, we get the Khovanskii-Teissier inequality
\begin{align*}
B_1 \cdot B_2 \cdot\cdot\cdot B_n \geq 1.
\end{align*}
And we know the equality $B_1 \cdot B_2 \cdot\cdot\cdot B_n =1$ holds if and only if the K\"ahler metrics $T_j$ are pointwise proportional. At this step, we can not conclude that the K\"ahler metrics $T_j$ are equal over $\Omega$ since we can not control the proportionality constants from the pointwise Brunn-Minkowski inequalities. However, for any pair of $T_i$ and $T_j$, we have the following pointwise equality over $\Omega$:
\begin{align*}
T_i ^{n-1} \wedge T_j = \left(\frac{T_i ^n}{\Phi}\right)^{n-1/n}\cdot \left(\frac{T_j ^n}{\Phi}\right)^{1/n} \Phi,
\end{align*}
since $T_i$ and $T_j$ are pointwise proportional over $\Omega$. This implies the equality
\begin{align*}
B_i ^{n-1} \cdot B_j = 1.
\end{align*}
Then by the pointwise estimates of \cite{FX14}, we know the currents $T_i$ and $T_j$ must be equal over $X$, which implies $B_i = B_j$.

In conclusion, we get that
$B_1 \cdot B_2 \cdot\cdot\cdot B_n = 1$
if and only if the $B_j$ are equal.
\end{proof}

\subsection{Complete intersection cone}
Since the complete intersection cone plays an important role in the paper, we quickly outline its basic properties.  Recall that $\CI_{1}(X)$ is the closure of the set of all curve classes of the form $A^{n-1}$ for an ample divisor $A$.  It naturally has the structure of a closed pointed cone.

\begin{prop}
\label{prop boundary CI}
Let $X$ be a projective variety of dimension $n$.  Suppose that $\alpha \in \CI_{1}(X)$ lies on the boundary of the cone. Then either
\begin{enumerate}
\item $\alpha = B^{n-1}$ for some big and nef divisor class $B$, or
\item $\alpha$ lies on the boundary of $\Eff_{1}(X)$.
\end{enumerate}
\end{prop}

\begin{proof}
We fix an ample divisor class $K$. Since $\alpha \in \CI_{1}(X)$ is a boundary point of the cone, we can write $\alpha$ as the limit of classes $A_{i}^{n-1}$ for some sequence of ample divisor classes $A_{i}$.

First suppose that the values of $A_{i} \cdot K^{n-1}$ are bounded above as $i$ varies. Then the classes of the divisor $A_{i}$ vary in a compact set, so they have some nef accumulation point $B$. Clearly $\alpha = B^{n-1}$. Furthermore, if $B$ is not big then $\alpha$ will lie on the boundary of $\Eff_{1}(X)$ since in this case $B^{n-1} \cdot B = 0$.  If $B$ is big, then it is not ample, since the map $A \mapsto A^{n-1}$ from the ample cone of divisors to $N_{1}(X)$ is locally surjective.  Thus in this case $B$ is big and nef.

Now suppose that the values of $A_{i} \cdot K^{n-1}$ do not have any upper bound. Since the $A_{i}^{n-1}$ limit to $\alpha$, for $i$ sufficiently large we have
\begin{equation*}
2(\alpha \cdot K) > A_{i}^{n-1} \cdot K \geq \vol(A_{i})^{n-1/n} \vol(K)^{1/n}
\end{equation*}
by the Khovanskii-Teissier inequality. In particular this shows that $\vol(A_{i})$ admits an upper bound as $i$ varies.
Note that the classes $A_{i}/(K^{n-1} \cdot A_{i})$ vary in a compact slice of the nef cone of divisors. Without loss of generality, we can assume they limit to a nef divisor class $B$. Then we have
\begin{align*}
B \cdot \alpha & = \lim_{i \to \infty} \frac{A_{i}}{K^{n-1} \cdot A_{i}} \cdot A_{i}^{n-1} \\
& =  \lim_{i \to \infty} \frac{\vol(A_{i})}{K^{n-1} \cdot A_{i}} \\
& = 0.
\end{align*}
The last equality holds because $\vol(A_{i})$ is bounded above but $A_{i} \cdot K^{n-1}$ is not. So in this case $\alpha$ must be on the boundary of the pseudo-effective cone $\Eff_1$.
\end{proof}

The complete intersection cone differs from most cones considered in birational geometry in that it is \emph{not} convex.  Since we are not aware of any such example in the literature, we give a toric example from \cite{fs09} in the appendix.  The same example shows that the cone that is the closure of all products of $(n-1)$ ample divisors is also not convex.

\begin{rmk}
\label{rmk CI}
It is still true that $\CI_{1}(X)$ is ``locally convex". Let $A, B$ be two ample divisor classes. If $\epsilon$ is sufficiently small, then $$A^{n-1}+\epsilon B^{n-1}=A_\epsilon ^{n-1}$$
for a unique ample divisor $A_\epsilon$. The existence of $A_\epsilon$ follows from the Hard Lefschetz theorem. Consider the following smooth map
\begin{align*}
\Phi: N^1(X) \rightarrow N_1(X)
\end{align*}
sending $D$ to $D^{n-1}$. By the Hard Lefschetz theorem, the derivative $d\Phi$ is an isomorphism at the point $A$. Thus $\Phi$ is local diffeomorphism near $A$, yielding the existence of $A_\epsilon$. The uniqueness follows from Teissier's proportionality theorem.  (See \cite{GT13} for a more in-depth discussion.)
\end{rmk}

Another natural question is:

\begin{ques}
Suppose that $X$ is a projective variety of dimension $n$ and that $\{ A_{i} \}_{i=1}^{n-1}$ are ample divisor classes on $X$.  Then is $A_{1} \cdot \ldots \cdot A_{n-1} \in \CI_{1}(X)$?
\end{ques}

One can imagine that such a statement may be studied using an ``averaging'' method. We hope Theorem \ref{thm KT inequality} would be helpful in the study of this problem.

\subsection{Fields of characteristic $p$} \label{charp background sec}
Almost all the results in the paper will hold for smooth varieties over an arbitrary algebraically closed field.  The necessary technical generalizations are verified in the following references:
\begin{itemize}
\item \cite[Remark 1.6.5]{lazarsfeld04} checks that the Khovanskii-Teissier inequalities hold over an arbitrary algebraically closed field.
\item The existence of Fujita approximations over an arbitrary algebraically closed field is proved in \cite{takagi07}.
\item The basic properties of the $\sigma$-decomposition in positive characteristic are considered in \cite{mustata11}.
\item The results of \cite{Cut13} lay the foundations of the theory of positive products and volumes over an arbitrary field.
\item \cite{fl14} describes how the above results can be used to extend \cite{BDPP13} and most of the results of \cite{bfj09} over an arbitrary algebraically closed field.  In particular the description of the derivative of the volume function in \cite[Theorem A]{bfj09} holds for smooth varieties in any characteristic.
\end{itemize}

\subsection{Compact K\"ahler manifolds}
\label{kahler background sec}
The following results enable us to extend most of our results to arbitrary compact K\"ahler manifolds.
\begin{itemize}
\item The Khovanskii-Teissier inequalities for classes in the nef cone $\overline{\mathcal{K}}$ can be proved by the mixed Hodge-Riemann bilinear relations \cite{DN06},
    or by solving complex Monge-Amp\`{e}re equations \cite{Dem93}; see also Theorem \ref{thm KT inequality}.
\item Teissier's proportionality theorem for transcendental big and nef classes has recently been proved by \cite{FX14}; see also Theorem \ref{thm KT inequality}.
\item The theory of positive intersection products for pseudo-effective $(1,1)$-classes has been developed by \cite{Bou02, BDPP13, BEGZ10MAbig}.
\item The cone duality $\overline{\mathcal{K}}^* =\mathcal{N}$ follows from the numerical characterization of the K\"ahler cone of \cite{DP04}.
\end{itemize}
We remark that we need the cone duality $\overline{\mathcal{K}}^* =\mathcal{N}$ to extend the Zariski decompositions and Morse-type inequality for curves to positive currents of bidimension $(1,1)$.

Comparing with the projective situation, the main ingredient missing is
Demailly's conjecture on the transcendental holomorphic Morse inequality, which is in turn implied by the expected identification of the derivative of the volume function on pseudo-effective $(1,1)$-classes as in \cite{bfj09}. Indeed, it is not hard to see these two expected results are equivalent (see e.g. \cite[Proposition 1.1]{xiao2014movable} -- which is essentially \cite[Section 3.2]{bfj09}). And they would imply the duality of the cones $\mathcal{M}(X)$ and $\mathcal{E}(X)$.  Thus, any of our results which relies on either
the transcendental holomorphic Morse inequality, or the results of \cite{bfj09}, is still conjectural in the K\"ahler setting.
However, these conjectures are known if $X$ is a compact hyperk\"ahler manifold (see \cite[Theorem 10.12]{BDPP13}), so all of our results extend to compact hyperk\"ahler manifolds.

\section{Polar transforms}  \label{legendresection}

As explained in the introduction, Zariski decompositions capture the failure of the volume function to be strictly log concave.  In this section and the next, we use some basic convex analysis to define a formal Zariski decomposition which makes sense for any non-negative homogeneous log concave function on a cone.  The main tool is a Legendre-Fenchel type transform for such functions.

\subsection{Duality transforms}

Let $V$ be a finite-dimensional $\mathbb{R}$-vector space of dimension $n$,  and let $V^{*}$ be its dual. We denote the pairing of $w^* \in V^*$ and $v \in V$ by $w^* \cdot v$.  Let $\Cvx(V)$ denote the class of lower-semicontinuous convex functions on $V$.  Then \cite[Theorem 1]{milman09legendre} shows that, up to composition with an additive linear function and a symmetric linear transformation, the Legendre-Fenchel transform is the unique order-reversing involution $\mathcal{L}: \Cvx(V) \to \Cvx(V^{*})$.  Motivated by this result, the authors define a duality transform to be an order-reversing involution of this type and characterize the duality transforms in many other contexts (see e.g. \cite{milman11hiddenduality}, \cite{milman08sconcave}).

In this section we study a duality transform for the set of non-negative homogeneous functions on a cone.
This transform is the concave homogeneous version of the well-known polar transform; see \cite[Chapter 15]{rockafellar70convexBOOK} for the basic properties of this transform in a related context.
This transform is also a special case of the generalized Legendre-Fenchel transform studied by \cite[Section 14]{Moreau1966-1967}, which is the usual Legendre-Fenchel transform with a ``coupling function'' -- we would like to thank M. Jonsson for pointing this out to us.  See also \cite[Section 0.6]{singerconvexBOOK} and \cite[Chapter 1]{rubinov00} for a brief introduction to this perspective.
Finally, it is essentially the same as the transform $\mathcal{A}$ from \cite{milman11hiddenduality} when applied to homogeneous functions, and is closely related to other constructions of \cite{milman08sconcave}. \cite[Chapter 2]{rubinov00} and \cite{dr02} work in a different setting which nonetheless has some nice parallels with our situation.

Let $\mathcal{C}\subset V$ be a proper closed convex cone of full dimension and let $\mathcal{C}^* \subset V^*$ denote the dual cone of $\mathcal{C}$, that is,
\begin{align*}
\mathcal{C}^* =\{w^* \in V^*|\ w^* \cdot v \geq 0 \ \textrm{for any}\ v\in \mathcal{C}\}.
\end{align*}
We let $\HConc_s (\mathcal{C})$ denote the collection of functions $f: \mathcal{C} \to \mathbb{R}$ satisfying:
\begin{itemize}
  \item $f$ is upper-semicontinuous and homogeneous of weight $s>1$;
  \item $f$ is strictly positive in the interior of $\mathcal{C}$ (and hence non-negative on $\mathcal{C}$);
  \item $f$ is $s$-concave: for any $v,x \in \mathcal{C}$ we have $f(v)^{1/s} + f(x)^{1/s} \leq f(v+x)^{1/s}$.
\end{itemize}
Note that since $f^{1/s}$ is homogeneous of degree $1$, the definition of concavity for $f^{1/s}$ above coheres with the usual one. For any $f\in \HConc_s (\mathcal{C})$, the function $f^{1/s}$ can extend to a proper upper-semicontinuous concave function over $V$ by letting $f^{1/s}(v)=-\infty$ whenever $v\notin \mathcal{C}$.  Thus many tools developed for arbitrary concave functions on $V$ also apply in our case.


Since an upper-semicontinuous function is continuous along decreasing sequences, the following continuity property of $f$ follows immediately from the non-negativity and concavity of $f^{1/s}$.

\begin{lem} \label{uppersemilimit}
Let $f \in \HConc_{s}(\mathcal{C})$ and $v \in \mathcal{C}$.  For any element $x \in \mathcal{C}$ we have $$f(v) = \lim_{t \to 0^{+}} f(v + tx).$$
\end{lem}

In particular, any $f \in \HConc_{s}(\mathcal{C})$ must vanish at the origin. \\

In this section we outline the basic properties of  the polar transform $\mathcal{H}$ (following a suggestion of M.~Jonsson).  In contrast to abstract convex transforms, $\mathcal{H}$ retains all of the properties of the classical Lengendre-Fenchel transform.  Since the proofs are essentially the same as in the theory of classical convex analysis, we omit most of the proofs in this section.

Recall that the polar transform $\mathcal{H}$ associates to a function $f \in \HConc_{s}(\mathcal{C})$ the function $\mathcal{H}f: \mathcal{C}^{*} \to \mathbb{R}$ defined as
\begin{align*}
\mathcal{H} f (w^*):= \inf_{v\in \mathcal{C}^{\circ}} \left(\frac{w^* \cdot v}{f(v)^{1/s}}\right)^{s/s-1}.
\end{align*}
By Lemma \ref{uppersemilimit} the definition is unchanged if we instead vary $v$ over all elements of $\mathcal{C}$ where $f$ is positive.  The following proposition shows that $\mathcal{H}$ defines an order-reversing involution from $\HConc_{s}(\mathcal{C})$ to $\HConc_{s/s-1}(\mathcal{C}^{*})$. Its proof is similar to the classical result in convex analysis,
see e.g. \cite[Theorem 15.1]{rockafellar70convexBOOK}.

\begin{prop}
\label{prop H involution}
Let $f,g \in \HConc_s(\mathcal{C})$.  Then we have
\begin{enumerate}
\item $\mathcal{H} f \in \HConc_{s/s-1}(\mathcal{C}^{*})$.
\item If $f \leq g$ then $\mathcal{H} f \geq \mathcal{H} g$.
\item $\mathcal{H} ^2 f  = f$.
\end{enumerate}
\end{prop}

It will be crucial to understand which points obtain the infimum in the definition of $\mathcal{H}f$.

\begin{defn}
Let $f \in \HConc_{s}(\mathcal{C})$.  For any $w^* \in \mathcal{C}^{*}$, we define $G_{w^*}$ to be the set of all $v \in \mathcal{C}$ which satisfy $f(v)>0$ and which achieve the infimum in the definition of $\mathcal{H}f(w^*)$, so that
\begin{align*}
\mathcal{H} f (w^*) = \left( \frac{w^* \cdot v}{f(v)^{1/s}} \right)^{s/s-1}.
\end{align*}
\end{defn}

\begin{rmk}
The set $G_{w^*}$ is the analogue of supergradients of concave functions. In particular, in the following sections we will see that the differential of $\mathcal{H}f$ at $w^*$ lies in $G_{w^*}$ if $\mathcal{H}f$ is differentiable.
\end{rmk}

It is easy to see that $G_{w^*} \cup \{ 0 \}$ is a convex subcone of $\mathcal{C}$.  Note the symmetry in the definition: if $v \in G_{w^{*}}$ and $\mathcal{H}f(w^{*})>0$ then $w^* \in G_{v}$.  Thus if $v \in \mathcal{C}$ and  $w^*\in \mathcal{C}^*$ satisfy $f(v)>0$ and $\mathcal{H}f(w^{*})>0$ then the conditions $v \in G_{w^{*}}$ and $w^{*} \in G_{v}$ are equivalent.

The analogue of the Young-Fenchel inequality in our situation is:

\begin{prop} \label{younginequality}
Let $f \in \HConc_{s}(\mathcal{C})$.  Then for any $v \in \mathcal{C}$ and $w^* \in \mathcal{C}^*$ we have
\begin{equation*}
\mathcal{H}f(w^*)^{s-1/s} f(v)^{1/s} \leq v \cdot w^*.
\end{equation*}
Furthermore, equality is obtained only if either $v \in G_{w^*}$ and $w^{*} \in G_v$, or at least one of $\mathcal{H}f(w^{*})$ and $f(v)$ vanishes.
\end{prop}

The next theorem describes the basic properties of $G_{v}$:

\begin{thrm}
\label{thm abstract zariski}
Let $f \in \HConc_{s}(\mathcal{C})$.
\begin{enumerate}
\item Fix $v \in \mathcal{C}$.  Let $\{w_{i}^{*}\}$ be a sequence of elements of $\mathcal{C}^{*}$ with $\mathcal{H}f(w_{i}^{*}) = 1$ such that $$f(v) = \lim_{i} (v \cdot w_{i}^{*})^{s} >0.$$  Suppose that the sequence admits an accumulation point $w^*$.  Then $f(v) = (v \cdot w^*)^s$ and $\mathcal{H}f(w^*)=1$.
\item For every $v \in \mathcal{C}^{\circ}$ we have that $G_{v}$ is non-empty.
\item Fix $v \in \mathcal{C}^{\circ}$.  Let $\{ v_{i} \}$ be a sequence of elements of $\mathcal{C}^{\circ}$ whose limit is $v$ and for each $v_{i}$ choose $w_{i}^{*} \in G_{v_{i}}$ with $\mathcal{H}f(w_{i}^{*}) = 1$.  Then the $w_{i}^{*}$ admit an accumulation point $w^*$, and any accumulation point lies in $G_{v}$ and satisfies $\mathcal{H}f(w^*)=1$.
\end{enumerate}
\end{thrm}

\begin{proof}
(1) The limiting statement for $f(v)$ is clear.  We have $\mathcal{H}f(w^*) \geq 1$ by upper semicontinuity, so that
\begin{equation*}
f(v)^{1/s} = \lim_{i \to \infty} v \cdot w_{i}^* \geq \frac{v \cdot w^*}{\mathcal{H}f(w^*)^{s-1/s}} \geq f(v)^{1/s}.
\end{equation*}
Thus we have equality everywhere.  If $\mathcal{H}f(w^{*})^{s-1/s} > 1$ then we obtain a strict inequality in the middle, a contradiction.

(2) Let $w_i^*$ be a sequence of points in $\mathcal{C}^{* \circ}$ with $\mathcal{H}f(w_i^*)=1$ such that $f(v) = \lim_{i \to \infty} (w_i^* \cdot v)^{s}$.  By (1) it suffices to see that the $w_{i}^{*}$ vary in a compact set.  But since $v$ is an interior point, the set of points which have intersection with $v$ less than $2f(v)^{1/s}$ is bounded.

(3) By (1) it suffices to show that the $w_{i}^*$ vary in a compact set.  For sufficiently large $i$ we have that $2v_{i} - v \in \mathcal{C}$.  By the log concavity of $f$ on $\mathcal{C}$ we see that $f$ must be continuous at $v$.  Thus for any fixed $\epsilon > 0$, we have for sufficiently large $i$
\begin{equation*}
w_{i}^{*} \cdot v \leq 2 w_{i}^{*} \cdot v_{i} \leq 2(1+\epsilon) f(v)^{1/s}.
\end{equation*}
Since $v$ lies in the interior of $\mathcal{C}$, this implies that the $w_{i}^{*}$ must lie in a bounded set.
\end{proof}

We next identify the collection of points where $f$ is controlled by $\mathcal{H}$.

\begin{defn}
Let $f \in \HConc_{s}(\mathcal{C})$.  We define $\mathcal{C}_{f}$ to be the set of all $v \in \mathcal{C}$ such that $v \in G_{w^*}$ for some $w^{*} \in \mathcal{C}$ satisfying $\mathcal{H}f(w^*)>0$.
\end{defn}

Since $v\in G_{w^*}$ and $\mathcal{H}f(w^*)>0$, Proposition \ref{younginequality} and the symmetry of $G$ show that $w^{*} \in G_{v}$.  Furthermore, we have $\mathcal{C}^{\circ} \subset \mathcal{C}_{f}$ by Theorem \ref{thm abstract zariski} and the symmetry of $G$.

\subsection{Differentiability}

\begin{defn}
We say that $f \in \HConc_{s}(\mathcal{C})$ is differentiable if it is $\mathcal{C}^{1}$ on $\mathcal{C}^{\circ}$.  In this case we define the function
\begin{align*}
D: \mathcal{C}^{\circ}  \to V^{*} \qquad \qquad \textrm{by} \qquad \qquad
v  \mapsto \frac{Df(v)}{s}.
\end{align*}
\end{defn}

The main properties of the derivative are:

\begin{thrm} \label{derivativeandbm}
Suppose that $f \in \HConc_{s}(\mathcal{C})$ is differentiable.  Then
\begin{enumerate}
\item $D$ defines an $(s-1)$-homogeneous function from $\mathcal{C}^{\circ}$ to $\mathcal{C}^{*}_{\mathcal{H}f}$.
\item $D$ satisfies a Brunn-Minkowski inequality with respect to $f$: for any $v \in \mathcal{C}^{\circ}$ and $x \in \mathcal{C}$
\begin{equation*}
D(v) \cdot x \geq f(v)^{s-1/s} f(x)^{1/s}.
\end{equation*}
Moreover, we have $D(v) \cdot v = f(v) = \mathcal{H}f(D(v))$.
\end{enumerate}
\end{thrm}



We will need the following familiar criterion for the differentiability of $f$, which is  an analogue of related results in convex analysis connecting the differentiability with the uniqueness of supergradient (see e.g. \cite[Theorem 25.1]{rockafellar70convexBOOK}).

\begin{prop} \label{diffuniqueness}
Let $f \in \HConc_{s}(\mathcal{C})$.  Let $U\subset \mathcal{C}^{\circ}$ be an open set.  Then $f|_{U}$ is differentiable if and only if for every $v \in U$ the set $G_{v} \cup \{0\}$ consists of a single ray.  In this case $D(v)$ is defined by intersecting against the unique element $w^{*} \in G_{v}$ satisfying $\mathcal{H}f(w^{*})=f(v)$.
\end{prop}

We next discuss the behaviour of the derivative along the boundary.

\begin{defn}
We say that $f \in \HConc_{s}(\mathcal{C})$ is $+$-differentiable if $f$ is $\mathcal{C}^{1}$ on $\mathcal{C}^{\circ}$
and the derivative on $\mathcal{C}^{\circ}$ extends to a continuous function on all of $\mathcal{C}_{f}$.
\end{defn}

It is easy to see that the $+$-differentiability implies continuity.
\begin{lem}\label{lem +c1 implies continuity}
If $f \in \HConc_{s}(\mathcal{C})$ is $+$-differentiable then $f$ is continuous on $\mathcal{C}_f$.
\end{lem}

\begin{rmk} \label{extensiontoboundaryrmk}
For $+$-differentiable functions $f$, we define the function $D: \mathcal{C}_{f} \to V^{*}$ by extending continuously from $\mathcal{C}^{\circ}$.
Many of the properties in Theorem \ref{derivativeandbm} hold for $D$ on all of $\mathcal{C}_f$.  By taking limits and applying Lemma \ref{uppersemilimit} we obtain the Brunn-Minkowski inequality. In particular, for any $x\in \mathcal{C}_f$ we still have
$$D(x) \cdot x = f(x)=\mathcal{H}f(D(x)).$$
Thus it is clear that $D(x)\in \mathcal{C}_{\mathcal{H}f}^{*}$ for any $x\in \mathcal{C}_f$.
\end{rmk}

\begin{lem}\label{lemma extention derivative}
Assume $f \in \HConc_{s}(\mathcal{C})$ is $+$-differentiable. For any $x\in \mathcal{C}_f$ and $y\in \mathcal{C}^\circ$, we have
\begin{align*}
\left. \frac{d}{dt}\right|_{t=0^+} f(x+ty)^{1/s}=(D(x)\cdot y) f(x)^{1-s/s}.
\end{align*}
\end{lem}


We next analyze what we can deduce about $f$ in a neighborhood of $v \in \mathcal{C}_{f}$ from the fact that $G_{v} \cup \{ 0 \}$ is a unique ray.

\begin{lem}
\label{lem compactness}
Let $f \in \HConc_{s}(\mathcal{C})$.  Let $v \in \mathcal{C}_{f}$  and assume that $G_{v} \cup \{0\}$ consists of a single ray.  Suppose $\{v_i\}$ is a sequence of elements of $\mathcal{C}_{f}$ converging to $v$. Let
$w^*_i \in G_{v_i}$ be any point satisfying $\mathcal{H}f(w^*_i)=1$.  Then the $w^*_{i}$ vary in a compact set.  Any accumulation point $w^*$ must be the unique point in $G_{v}$ satisfying $\mathcal{H}f(w^*)=1$.
\end{lem}

\begin{proof}
By Theorem \ref{thm abstract zariski} it suffices to prove that the $w^*_i$ vary in a compact set.  Otherwise, we must have that $w^*_i \cdot m$ is unbounded for some interior point $m \in \mathcal{C}^\circ$.  By passing to a subsequence we may suppose that $w^*_{i} \cdot m \to \infty$.  Consider the normalization
$$\widehat{w}_i^* := \frac{w^*_i}{w^*_i \cdot m};$$
note that $\widehat{w}_i^*$ vary in a compact set.
Take some convergent subsequence, which we still denote
by $\widehat{w}_i^*$, and write $\widehat{w}_i^* \rightarrow \widehat{w}_0 ^*$.  Since $\widehat{w}_0 ^*\cdot m =1$ we see that $\widehat{w}_0 ^* \neq 0$.

We first prove $v \cdot \widehat{w}_0 ^* >0$. Otherwise, $v \cdot \widehat{w}_0 ^* =0$ implies
\begin{align*}
\frac{v \cdot (w^*+ \widehat{w}_0 ^*)}{\mathcal{H}f(w^*+ \widehat{w}_0 ^*)^{s-1/s}} \leq
\frac{v \cdot w^*}{\mathcal{H}f(w^*)^{s-1/s}}
=f(v)^{1/s}.
\end{align*}
By our assumption on $G_{v}$, we get $w^*+ \widehat{w}_0 ^*$ and $w^*$ are proportional, which implies $\widehat{w}_0 ^*$ lies in the ray spanned by $w^*$. Since $\widehat{w}_0 ^* \neq 0$ and $v \cdot w^* >0$, we get that  $v \cdot \widehat{w}_0 ^* >0$.  So our assumption $v \cdot \widehat{w}_0 ^* =0$ does not hold.
On the other hand, $\mathcal{H}f(w_i ^*)=1$ implies
$$\mathcal{H}f(\widehat{w}_i ^*)^{s-1/s}
=\frac{1}{m \cdot w_i^*}\rightarrow 0 .$$
By the upper-semicontinuity of $f$ and the fact that $\lim v_i \cdot \widehat{w}_i ^* =v \cdot  \widehat{w}_0 ^* >0$, we get
\begin{align*}
f(v)^{1/s}&\geq \limsup_{i \to \infty} f(v_i)^{1/s}\\
& =\limsup_{i \to \infty} \frac{v_i \cdot \widehat{w}_i^*}{\mathcal{H}f(\widehat{w}_i ^*)^{s-1/s}}=\infty.
\end{align*}
This is a contradiction, thus the sequence $w^* _i$ must vary in a compact set.
\end{proof}

\begin{thrm} \label{seconddiffuniqueness}
Let $f \in \HConc_{s}(\mathcal{C})$.  Suppose that $U\subset \mathcal{C}_f$ is a relatively open set and $G_{v} \cup \{0\}$ consists of a single ray for any $v \in U$.  If $f$ is continuous on $U$ then $f$ is $+$-differentiable on $U$.  In this case $D(v)$ is defined by intersecting against the unique element $w^{*} \in G_{v}$ satisfying $\mathcal{H}f(w^{*})=f(v)$.
\end{thrm}

Even if $f$ is not continuous, we at least have a similar statement along the directions in which $f$ is continuous (for example, any directional derivative toward the interior of the cone).

\begin{proof}
Theorem \ref{diffuniqueness} shows that $f$ is differentiable on $U \cap \mathcal{C}^{\circ}$ and is determined by intersections.  By combining Lemma \ref{lem compactness} with the continuity of $f$, we see that the derivative extends continuously to any point in $U$.
\end{proof}

\begin{rmk}\label{rmk not a single ray}
Assume $f\in \HConc_s (\mathcal{C})$ is $+$-differentiable. In general, we can not conclude that $G_{v}\cup \{0\}$ contains a single ray if $x\in \mathcal{C}_f$ is not an interior point. An explicit example is in Section \ref{section zariski}. Let $X$ be a smooth projective variety of dimension $n$, let $\mathcal{C}=\Nef^1 (X)$ be the cone of nef divisor classes and let $f=\vol$ be the volume function of divisors. Let $B$ be a big and nef divisor class which is not ample.  Then $G_B$ contains the cone generated by all $B^{n-1}+ \gamma$ with $\gamma$ pseudo-effective and $B\cdot \gamma=0$, which in general is more than a ray.
\end{rmk}

\section{Formal Zariski decompositions}
\label{formal zariski section}

The Legendre-Fenchel transform relates the strict concavity of a function to the differentiability of its transform.  The transform $\mathcal{H}$ will play the same role in our situation; however, one needs to interpret the strict concavity slightly differently.  We will encapsulate this property using the notion of a Zariski decomposition.

\begin{defn}\label{definition formal zariski}
Let $f \in \HConc_{s}(\mathcal{C})$ and let $U \subset \mathcal{C}$ be a non-empty subcone.  We say that $f$ admits a strong Zariski decomposition with respect to $U$ if:
\begin{enumerate}
\item For every $v \in \mathcal{C}_{f}$ there are unique elements $p_{v} \in U$ and $n_{v} \in \mathcal{C}$ satisfying
\begin{equation*}
v = p_{v} + n_{v} \qquad \qquad \textrm{and} \qquad \qquad f(v) = f(p_{v}).
\end{equation*}
We call the expression $v = p_v + n_v$ the Zariski decomposition of $v$, and call $p_v$ the positive part and $n_v$ the negative part of $v$.
\item For any $v,w \in \mathcal{C}_f$ satisfying $v+w \in \mathcal{C}_f$ we have
\begin{equation*}
f(v)^{1/s} + f(w)^{1/s} \leq f(v+w)^{1/s}
\end{equation*}
with equality only if $p_{v}$ and $p_{w}$ are proportional.
\end{enumerate}
\end{defn}

\begin{rmk}
Note that the vector $n_{v}$ must satisfy $f(n_{v})=0$  by the non-negativity and log-concavity of $f$.  In particular $n_{v}$ lies on the boundary of $\mathcal{C}$.  Furthermore, any $w^{*} \in G_{v}$ is also in $G_{p_{v}}$ and must satisfy $w^{*} \cdot n_{v} = 0$.

Note also that the proportionality of $p_{v}$ and $p_{w}$ may not be enough to conclude that $f(v)^{1/s} + f(w)^{1/s} = f(v+w)^{1/s}$.  This additional property turns out to rely on the strict log concavity of $\mathcal{H}f$.
\end{rmk}

The main principle of the section is that when $f$ satisfies a differentiability property, $\mathcal{H}f$ admits some kind of Zariski decomposition.  Usually the converse is false, due to the asymmetry of $G$ when $f$ or $\mathcal{H}f$ vanishes.  However, the existence of a Zariski decomposition is usually strong enough to determine the differentiability of $f$ along some subcone.  We will give a version that takes into account the behavior of $f$ along the boundary of $\mathcal{C}$.

\begin{thrm} \label{strong zariski equivalence}
Let $f \in \HConc_{s}(\mathcal{C})$.
Then we have the following results:
\begin{itemize}
\item If $f$ is $+$-differentiable, then $\mathcal{H}f$ admits a strong Zariski decomposition  with respect to the cone $D(\mathcal{C}_{f}) \cup \{0\}$.
\item If $\mathcal{H}f$ admits a strong Zariski decomposition with respect to a cone $U$, then $f$ is differentiable.
\end{itemize}
\end{thrm}

\begin{proof}
First suppose $f$ is $+$-differentiable; we must prove the function $\mathcal{H}f$ satisfies properties $(1), (2)$ in Definition \ref{definition formal zariski}.   

We first show the existence of the Zariski decomposition in property (1).  If $w^* \in \mathcal{C}^{*}_{\mathcal{H}f}$ then by definition there is some $v \in \mathcal{C}$ satisfying $f(v)>0$ such that $w^{*} \in G_{v}$. In particular, by the symmetry of $G$ we also have $v\in G_{w^*}$, thus $v\in \mathcal{C}_f$.
Since $f(v)>0$ we can define
\begin{equation*}
p_{w^*} := \left( \frac{\mathcal{H}f(w^*)}{f(v)} \right)^{s-1/s} \cdot D(v), \qquad \qquad n_{w^*} = w^* - p_{w^*}.
\end{equation*}
Then $p_{w^*} \in D(\mathcal{C}_f)$ and
\begin{align*}
\mathcal{H}f(p_{w^*}) & = \mathcal{H} \left(\left( \frac{\mathcal{H}f(w^*)}{f(v)} \right)^{s-1/s} \cdot D(v)  \right) \\
& = \frac{\mathcal{H}f(w^*)}{f(v)}  \cdot \mathcal{H}f \left (D(v) \right) = \mathcal{H}f(w^*)
\end{align*}
where the final equality follows from Theorem \ref{derivativeandbm} and Remark \ref{extensiontoboundaryrmk}.
We next show that $n_{w^*} \in \mathcal{C}^*$.  Choose any $x \in \mathcal{C}^\circ$ and note that for any $t>0$ we have the inequality
\begin{equation*}
\frac{v + tx}{f(v+tx)^{1/s}} \cdot w^* \geq \frac{v}{f(v)^{1/s}} \cdot w^*
\end{equation*}
with equality when $t=0$.
By Lemma \ref{lemma extention derivative},
taking derivatives at $t=0$ we obtain
\begin{equation*}
\frac{x \cdot w^*}{f(v)^{1/s}} - \frac{(v \cdot w^*)(D(v) \cdot x)}{f(v)^{(s+1)/s}} \geq 0,
\end{equation*}
or equivalently, identifying $v \cdot w^*/f(v)^{1/s} = \mathcal{H}f(w^*)^{s-1/s}$,
\begin{equation*}
x \cdot \left(w^* -D(v) \cdot \frac{\mathcal{H}f(w^*)^{s-1/s}}{f(v)^{s-1/s}} \right) \geq 0.
\end{equation*}
Since this is true for any $x \in \mathcal{C}^\circ$, we see that $n_{w^*} \in \mathcal{C}^*$ as claimed.

We next show that $p_{w^*}$ constructed above is the unique element of $D(\mathcal{C}_{f})$ 
satisfying the two given properties. First,  after some rescaling we can assume $\mathcal{H}f(w^*)=f(v)$, which then implies $w^*\cdot v=f(v)$.
Suppose that $z \in \mathcal{C}_{f}$ and $D(z)$ is another vector satisfying $\mathcal{H}f(D(z)) = \mathcal{H}f(w^*)$ and $w^* - D(z) \in \mathcal{C}$.  Note that by Remark \ref{extensiontoboundaryrmk} $f(z) = \mathcal{H}f(D(z))= f(v)$.  By Proposition \ref{younginequality} we have
\begin{equation*}
\mathcal{H}f(D(z))^{s-1/s} f(v)^{1/s} \leq D(z) \cdot v \leq w^{*} \cdot v = f(v)
\end{equation*}
so we obtain equality everywhere.  In particular, we have $D(z)\cdot v=f(v)$.  By Theorem \ref{derivativeandbm}, for any $x \in \mathcal{C}$ we have
\begin{align*}
D(z) \cdot x  \geq f(z)^{s-1/s} f(x)^{1/s}.
\end{align*}
Set $x = v +\epsilon q$ where $\epsilon > 0$ and $q \in \mathcal{C}^{\circ}$.  With this substitution, the two sides of the equation above are equal at $\epsilon = 0$, so taking an $\epsilon$-derivative of the above equation and arguing as before, we see that $D(z) - D(v) \in \mathcal{C}^{*}$.

We claim that $D(z)=D(v)$. First we note that $D(v)\cdot z= f(z)$. Indeed, since $f(z)=f(v)$ and $D(v)\preceq D(z)$ we have
\begin{align*}
f(v)^{s-1/s} f(z)^{1/s}\leq D(v)\cdot z \leq D(z)\cdot z =f(z).
\end{align*}
Thus we have equality everywhere, proving the equality $D(v)\cdot z= f(z)$.  Then we can apply the same argument as before with the roles of $v$ and $z$ switched.  This shows $D(v)\succeq D(z)$, so we must have $D(z)=D(v)$.

We next turn to (2).  The inequality is clear, so we only need to characterize the equality. Suppose $w^*, y^* \in \mathcal{C}_{\mathcal{H}f} ^{*}$ satisfy
\begin{equation*}
\mathcal{H}f(w^*)^{s-1/s} + \mathcal{H}f(y^*)^{s-1/s} = \mathcal{H}f(w^* + y^*)^{s-1/s}
\end{equation*}
and $w^* + y^* \in \mathcal{C}_{\mathcal{H}f} ^{*}$.  We need to show they have proportional positive parts.
By assumption $G_{w^*+y^*}$ is non-empty, so we may choose some $v \in G_{w^*+y^*}$.  Then also $v \in G_{w^*}$ and $v \in G_{y^*}$.
Note that by homogeneity $v$ is also in $G_{aw^{*}}$ and $G_{by^{*}}$ for any positive real numbers $a$ and $b$.  Thus by rescaling $w^*$ and $y^*$, we may suppose that both have intersection $f(v)$ against $v$, so that $\mathcal{H}f(w^{*}) = \mathcal{H}f(y^{*}) = f(v)$. Then we need to verify the positive parts of $w^*$ and $y^*$ are equal.  But they both coincide with $D(v)$ by the argument in the proof of (1).\\

Conversely, suppose that $\mathcal{H}f$ admits a strong Zariski decomposition with respect to the cone
$U$.  We claim that $f$ is differentiable.  By Proposition \ref{diffuniqueness} it suffices to show that $G_{v} \cup \{0\}$ is a single ray for any $v \in \mathcal{C}^{\circ}$.

For any two elements $w^{*}, y^{*}$ in $G_{v}$ we have
\begin{align*}
\mathcal{H}f(w^*)^{1/s} + \mathcal{H}f(y^{*})^{1/s} = \frac{w^{*} \cdot v}{f(v)^{1/s}} + \frac{y^{*} \cdot v}{f(v)^{1/s}} \geq \mathcal{H}f(w^{*}+y^{*})^{1/s}.
\end{align*}
Since $w^{*}$, $y^{*}$ and their sum are all in $\mathcal{C}^{*}_{\mathcal{H}f}$, we conclude by the strong Zariski decomposition condition that $w^{*}$ and $y^{*}$ have proportional positive parts.  After rescaling so that $\mathcal{H}f(w^{*}) = f(v) = \mathcal{H}f(y^{*})$ we have $p_{w^{*}} = p_{y^{*}}$. Thus it suffices to prove $w^* =p_{w^*}$. Note that $\mathcal{H}f(w^{*})=\mathcal{H}f(p_{w^{*}})$ as $p_{w^*}$ is the positive part. If $w^* \neq p_{w^*}$, then $v\cdot w^* > v\cdot p_{w^*}$ since $v$ is an interior point. This implies
\begin{align*}
f(v)=\inf_{y^* \in \mathcal{C}^{*\circ}}\left(\frac{v\cdot y^*}{\mathcal{H}f(y^*)^{s-1/s}}\right)^s<\left(\frac{v\cdot w^*}{\mathcal{H}f(w^*)^{s-1/s}}\right)^s,
\end{align*}
contradicting with $w^* \in G_v$.
Thus $w^*=p_{w^*}$ and $G_{v} \cup \{0\}$ must be a single ray.
\end{proof}

\begin{rmk}
It is worth emphasizing that if $f$ is $+$-differentiable and $w^{*} \in \mathcal{C}^{*}_{\mathcal{H}f}$, we can construct a positive part for $w^{*}$ by choosing \emph{any} $v \in G_{w^*}$ with $f(v) > 0$ and taking an appropriate rescaling of $D(v)$.
\end{rmk}

\begin{rmk}
It would also be interesting to study some kind of weak Zariski decomposition. For example, one can define a weak Zariski decomposition as a decomposition $v=p_v + n_v$ only demanding $f(v)=f(p_v)$ and the strict log concavity of $f$ over the set of positive parts.  Appropriately interpreted, the existence of a weak decomposition for $\mathcal{H}f$ should be a consequence of the differentiability of $f$.
\end{rmk}

Under some additional conditions, we can get the continuity of the Zariski decompositions.

\begin{thrm}\label{thrm positive part continuity}
Let $f \in \HConc_{s}(\mathcal{C})$ be $+$-differentiable.  Then the function taking an element $w^{*} \in \mathcal{C}^{* \circ}$ to its positive part $p_{w^*}$ is continuous.

If furthermore $G_{v} \cup \{ 0 \}$ is a unique ray for every $v \in \mathcal{C}_{f}$ and $\mathcal{H}f$ is continuous on all of $\mathcal{C}^{*}_{\mathcal{H}f}$, then the Zariski decomposition is continuous on all of $\mathcal{C}^{*}_{\mathcal{H}f}$.
\end{thrm}

\begin{proof}
Fix any $w^{*} \in \mathcal{C}^{* \circ}$ and suppose that $w^*_i$ is a sequence whose limit is $w^*$.  For each choose some $v_{i} \in G_{w^*_i}$ with $f(v_i) = 1$.  By Theorem \ref{thm abstract zariski}, the $v_{i}$ admit an accumulation point $v \in G_{w^*}$ with $f(v) = 1$.  By the symmetry of $G$, each $v_{i}$ and also $v$ lies in $\mathcal{C}_{f}$.
The $D(v_i)$ limit to $D(v)$ by the continuity of $D$. Recall that by the argument in the proof of Theorem \ref{strong zariski equivalence} we have $p_{w^*_i} = \mathcal{H}f(w_{i}^{*})^{s-1/s}D(v_{i})$ and similarly for $w^{*}$.  Since $\mathcal{H}f$ is continuous at interior points, we see that the positive parts vary continuously as well.

The last statement follows by a similar argument using Lemma \ref{lem compactness}.
\end{proof}

\begin{exmple} \label{bilinearexample}
Suppose that $q$ is a bilinear form on $V$ and $f(v) = q(v,v)$.  Let $\mathcal{P}$ denote one-half of the positive cone of vectors satisfying $f(v) \geq 0$.  It is easy to see that $f$ is $2$-concave and non-trivial on $\mathcal{P}$ if and only if $q$ has signature $(1,\dim V-1)$.  Identifying $V$ with $V^{*}$ under $q$, we have $\mathcal{P} = \mathcal{P}^{*}$ and $\mathcal{H}f = f$ by the usual Hodge inequality argument.

Now suppose $\mathcal{C} \subset \mathcal{P}$.  Then $\mathcal{C}^{*}$ contains $\mathcal{C}$.  As discussed above, by the Hodge inequality $\mathcal{H}f|_{\mathcal{C}} = f$.  Note that $f$ is everywhere differentiable and $D(v) = v$ for classes in $\mathcal{C}$.  Thus on $\mathcal{C}$ the polar transform $\mathcal{H}f$ agrees with $f$, but outside of $\mathcal{C}$ the function $\mathcal{H}f$ is controlled by a Zariski decomposition involving a projection to $\mathcal{C}$.

This is of course just the familiar picture for curves on a surface identifying $f$ with the self-intersection on the nef cone and $\mathcal{H}f$ with the volume on the pseudo-effective cone.  More precisely, for big curve classes the decomposition constructed in this way is the numerical version of Zariski's original construction.  Along the boundary of $\mathcal{C}^{*}$, the function $\mathcal{H}f$ vanishes identically so that Theorem \ref{strong zariski equivalence} does not apply.  The linear algebra arguments of \cite{zariski62}, \cite{bauer09} give a way of explicitly constructing the vector computing the minimal intersection as above.
\end{exmple}

\begin{exmple}
Fix a spanning set of unit vectors $\mathcal{Q}$ in $\mathbb{R}^n$.  Recall that the polytopes whose unit facet normals are a subset of $\mathcal{Q}$ naturally define a cone $\mathcal{C}$ in a finite dimensional vector space $V$ which parametrizes the constant terms of the bounding hyperplanes.  One can also consider the cone $\mathcal{C}_{\Sigma}$ which is the closure of those polytopes whose normal fan is $\Sigma$.  The volume function $\vol$ defines a weight-$n$ homogeneous function on $\mathcal{C}$ and (via restriction) $\vol_{\Sigma}$ on $\mathcal{C}_{\Sigma}$, and it is interesting to ask for the behavior of the polar transforms.  (Note that this is somewhat different from the link between polar sets and polar functions, which is described for example in \cite{milman11hiddenduality}.)

The dual space $V^{*}$ consists of the Minkowski weights on $\mathcal{Q}$.  We will focus on the subcone $\mathcal{M}$ of strictly positive Minkowski weights, which is contained in the dual of both cones.  By Minkowski's theorem, a strictly positive Minkowski weight determines naturally a polytope in $\mathcal{C}$, so we can identify $\mathcal{M}$ with the interior of $\mathcal{C}$.  As explained in Section \ref{toric section}, the Brunn-Minkowski inequality shows that $\mathcal{H}\vol|_{\mathcal{M}}$ coincides with the volume function on $\mathcal{M}$.  However, calculating $\mathcal{H}\vol_{\Sigma}|_{\mathcal{M}}$ is more subtle.

It would be very interesting to extend this duality to all convex sets, perhaps by working on an infinite dimensional space.
\end{exmple}

\subsection{Teissier proportionality}
In this section, we give some conditions which are equivalent to the strict log concavity. The prototype is the volume function of divisors over the cone of big and movable divisor classes.

\begin{defn}
Let $f \in \HConc_{s}(\mathcal{C})$ be $+$-differentiable and let $\mathcal{C}_{T}$ be a non-empty subcone of $\mathcal{C}_{f}$.  We say that $f$ satisfies Teissier proportionality with respect to $\mathcal{C}_{T}$ if for any $v,x \in \mathcal{C}_{T}$ satisfying
\begin{equation*}
D(v) \cdot x = f(v)^{s-1/s} f(x)^{1/s}
\end{equation*}
we have that $v$ and $x$ are proportional.
\end{defn}

Note that we do not assume that $\mathcal{C}_{T}$ is convex -- indeed, in examples it is important to avoid this condition.  However, since $f$ is defined on the convex hull of $\mathcal{C}_{T}$, we can (somewhat abusively) discuss the strict log concavity of $f|_{\mathcal{C}_{T}}$:

\begin{defn} \label{defn generalized log concave}
Let $\mathcal{C}' \subset \mathcal{C}$ be a (possibly non-convex) subcone.  We say that $f$ is strictly log concave on $\mathcal{C}'$ if
\begin{align*}
f(v)^{1/s} + f(x)^{1/s}< f(v+x)^{1/s}
\end{align*}
holds whenever $v, x\in \mathcal{C}'$ are not proportional.  Note that this definition makes sense even when $\mathcal{C}'$ is not itself convex.
\end{defn}

\begin{thrm}
\label{thrm teissier}
Let $f \in \HConc_{s}(\mathcal{C})$ be $+$-differentiable.  For any non-empty subcone $\mathcal{C}_{T}$ of $\mathcal{C}_{f}$, consider the following conditions:
\begin{enumerate}
\item The restriction $f|_{\mathcal{C}_{T}}$ is strictly log concave (in the sense defined above).
\item $f$ satisfies Teissier proportionality with respect to $\mathcal{C}_{T}$.
\item The restriction of $D$ to $\mathcal{C}_{T}$ is injective.
\end{enumerate}
Then we have (1) $\implies$ (2) $\implies$ (3).  If $\mathcal{C}_{T}$ is convex, then we have (2) $\implies$ (1).  If $\mathcal{C}_{T}$ is an open subcone, then we have (3) $\implies$ (1).
\end{thrm}

\begin{proof}
We first prove (1) $\implies$ (2). Let $v, x\in \mathcal{C}_T$ satisfy $D(v)\cdot x =f(v)^{s-1/s}f(x)^{1/s}$ and $f(v)=f(x)$.   Assume for a contradiction that $v\neq x$. Since $f|_{\mathcal{C}_T}$ is strictly log concave, for any two $v, x\in \mathcal{C}_T$ which are not proportional we have
\begin{align*}
f(x)^{1/s} < f(v)^{1/s} + \frac{D(v) \cdot (x-v)}{f(v)^{s-1/s}}.
\end{align*}
Since we have assumed $D(v)\cdot x =f(v)^{s-1/s}f(x)^{1/s}$ and $f(v)=f(x)$, we must have
\begin{align*}
f(x)^{1/s} = f(v)^{1/s} + \frac{D(v) \cdot (x-v)}{f(v)^{s-1/s}}
\end{align*}
since $D(v)\cdot v= f(v)$.  This is a contradiction, so we must have $v=x$. This then implies $f$ satisfies Teissier proportionality.

We next show (2) $\implies$ (3).  Let $v_1, v_2 \in \mathcal{C}_{T}$ with $D(v_1)=D(v_2)$. Then we have
\begin{align*}
f(v_1) &= D(v_1)\cdot v_1 =D(v_2)\cdot v_1\\
&\geq f(v_2)^{s-1/s}f(v_1)^{1/s},
\end{align*}
which implies $f(v_1)\geq f(v_2)$. By symmetry, we get $f(v_1)=f(v_2)$. So we must have
$$D(v_1)\cdot v_2= f(v_1)^{s-1/s}f(v_2)^{1/s}.$$
By the Teissier proportionality we see that $v_1, v_2$ are proportional, and since $f(v_1)=f(v_2)$ they must be equal.

We next show that if $\mathcal{C}_{T}$ is convex then (2) $\implies$ (1).  Fix $y$ in the interior of $\mathcal{C}_{T}$ and fix $\epsilon >0$.  Then
\begin{align*}
f(v+x+\epsilon y)^{1/s}-f(v)^{1/s}&=\int_0 ^1 (D(v+t(x+\epsilon y))\cdot x) f(v+t(x+\epsilon y))^{1-s/s}dt.
\end{align*}
The integrand is bounded by a positive constant independent of $\epsilon$ as we let $\epsilon$ go to $0$ due to the $+$-differentiability of $f$ (which also implies the continuity of $f$).  Using Lemma \ref{uppersemilimit}, the dominanted convergence theorem shows that
\begin{align*}
f(v+x)^{1/s}-f(v)^{1/s}&=\int_0 ^1 (D(v+tx)\cdot x) f(v+tx)^{1-s/s}dt.
\end{align*}
This immediately shows the strict log concavity.

Finally, we show that if $\mathcal{C}_{T}$ is open then (3) $\implies$ (1).  By \cite[Corollary 26.3.1]{rockafellar70convexBOOK}, it is clear that for any convex open set $U \subset \mathcal{C}_{T}$ the injectivity of $D$ over $U$ is equivalent to the strict log concavity of  $f|_{U}$.  Using the global log concavity of $f$, we obtain the conclusion. More precisely, assume $x, y\in \mathcal{C}_T$ are not proportional, then by the strict log concavity of $f$ near  $x$ and the global log concavity on $\mathcal{C}$, for $t>0$ sufficiently small we have
\begin{align*}
f^{1/s}(x+y)&\geq f^{1/s}(x+ty) +(1-t)f^{1/s}(y)\\
&> (f^{1/s}(x)+f^{1/s}(x+2ty))/2 + (1-t)f^{1/s}(y)\\
&\geq f^{1/s}(x)+f^{1/s}(y).
\end{align*}
\end{proof}

Another useful observation is:

\begin{prop} \label{derivativeofdual}
Let $f \in \HConc_{s}(\mathcal{C})$ be differentiable and suppose that $f$ is strictly log concave on an open subcone $\mathcal{C}_{T}\subset \mathcal{C}^{\circ}$.  Then $\mathcal{H}f$ is differentiable on $D(\mathcal{C}_{T})$ and the derivative is determined by the prescription
\begin{equation*}
D(D(v)) = v.
\end{equation*}
\end{prop}

\begin{proof}
We first show that $D(\mathcal{C}_{T}) \subset \mathcal{C}^{* \circ}$.  Suppose that there were some $v \in \mathcal{C}_{T}$ such that $D(v)$ lay on the boundary of $\mathcal{C}^{*}$.  Choose $x \in \mathcal{C}$ satisfying $x \cdot D(v) = 0$.  By openness we have $v + tx \in \mathcal{C}_{T}$ for sufficiently small $t$.  Since $D(v) \in G_{v + tx}$, we must have that $D(v)$ and $D(v+tx)$ are proportional by Proposition \ref{diffuniqueness}. This is a contradiction by Theorem \ref{thrm teissier}.

Now suppose $w^{*} = D(v) \in D(\mathcal{C}_{T})$.  By the strict log concavity of $f$ on $\mathcal{C}_{T}$ (and the global log concavity), we must have that $G_{w^*} \cup \{0 \}$ consists only of the ray spanned by $v$.  Applying Proposition \ref{diffuniqueness}, we obtain the statement.
\end{proof}

Combining all the results above, we obtain a very clean property of $D$ under the strongest possible assumptions.

\begin{thrm} \label{strongesttransform}
Assume $f \in \HConc_{s}(\mathcal{C})$ and its polar transform $\mathcal{H}f\in \HConc_{s/s-1}(\mathcal{C}^*)$ are $+$-differentiable. Let $U=D(\mathcal{C}_{\mathcal{H}f} ^*)\cup \{0\}$ and $U^*=D(\mathcal{C}_{f} )\cup \{0\}$. Then we have:
\begin{itemize}
\item $f$ and $\mathcal{H}f$ admit a strong Zariski decomposition with respect to the cone $U$ and the cone $U^*$ respectively;

\item For any $v\in \mathcal{C}_{f}$ we have $D(v)=D(p_v)$ (and similarly for $w \in \mathcal{C}_{\mathcal{H}f}^{*}$);

\item $D$ defines a bijection $D: U^{\circ} \to U^{* \circ}$ with inverse also given by $D$.  In particular, $f$ and $\mathcal{H}f$ satisfy Teissier proportionality with respect to the open cone $U^{\circ}$ and $U^{* \circ}$ respectively.

\end{itemize}
\end{thrm}

\begin{proof}
Note that $U^* \subset \mathcal{C}^{*}_{\mathcal{H}f}$ (and $U \subset \mathcal{C}_{f}$) since for any $v \in \mathcal{C}_{f}$ we have $D(v) \in G_{v}$ and $f(v) > 0$.

The first statement is immediate from Theorem \ref{strong zariski equivalence}.

We next show the second statement.  By the definition of positive parts, we have $G_v \subset G_{p_v}$.  Since both $v, p_{v} \in \mathcal{C}_{f}$, we know by the argument of Theorem \ref{strong zariski equivalence} that $D(v)$ and $D(p_{v})$ are both proportional to the (unique) positive part of any $w^{*} \in G_{v}$ with positive $\mathcal{H}f$.

Finally we show the third statement.  We start by proving the Teissier proportionality on $U^{\circ}$.  By part (2) of the Zariski decomposition condition $f$ is strictly log concave on $U^\circ$, and Teissier proportionality follows by Theorem \ref{thrm teissier}.  Furthermore, the argument of Proposition \ref{derivativeofdual} then shows that $D(U^{\circ})\subset \mathcal{C}^{* \circ}$ and $D(D(U^{\circ})) = U^{\circ}$.

We must show that $D(U^{\circ}) \subset U^{* \circ}$.  Suppose that $v \in U^{\circ}$ had that $D(v)$ was on the boundary of $U^{*}$.  Since $D(v) \in \mathcal{C}^{* \circ}$, there must be some sequence $w_{i}^{*} \in C^{* \circ} - U^{*}$ whose limit is $D(v)$.  We note that each $D(w_{i}^{*})$ lies on the boundary of $\mathcal{C}$,
thus must lie on the boundary of $U$.  Indeed, by the second statement we have $D(w_{i}^{*}) = D(w_{i}^{*} + tn_{w_{i}^{*}})$ for any $t>0$, which would violate the uniqueness of $G_{D(w_{i}^*)}$ as in Proposition \ref{diffuniqueness} if it were an interior point.  Using the continuity of $D$ we see that $v = D(D(v))$ lies on the boundary of $U$, a contradiction.

In all, we have shown that $D: U^{\circ} \to U^{* \circ}$ is an isomorphism onto its image with inverse $D$.  By symmetry we also have $D(U^{* \circ}) \subset U^{\circ}$, and we conclude after taking $D$ the reverse inclusion $U^{* \circ} \subset D(U^{\circ})$.
\end{proof}

\subsection{Morse-type inequality}

The polar transform $\mathcal{H}$ also gives a natural way of translating cone positivity conditions from $\mathcal{C}$ to $\mathcal{C}^{*}$.

\begin{defn}
Let $f \in \HConc_{s}(\mathcal{C})$ be $+$-differentiable.  We say that $f$ satisfies a Morse-type inequality if for any
$v\in \mathcal{C}_f$ and $x\in \mathcal{C}$ satisfying the inequality
$$
f(v)-sD(v)\cdot x >0
$$
we have that $v-x \in \mathcal{C}^\circ$.
\end{defn}

Note that the prototype of the Morse-type inequality is the well known algebraic Morse inequality for nef divisors.

In order to translate the positivity in $\mathcal{C}$ to $\mathcal{C}^*$, we need the following ``reverse" Khovanskii-Teissier inequality.

\begin{prop}
\label{prop abstract reverse KT inequality}
Let $f \in \HConc_{s}(\mathcal{C})$ be $+$-differentiable and satisfy a Morse-type inequality. Then we have
\begin{align*}
s(y^*\cdot v)(D(v)\cdot x)\geq f(v) (y^*\cdot x),
\end{align*}
for any $y^* \in \mathcal{C}^*$, $v\in \mathcal{C}_f$ and $x\in \mathcal{C}$.
\end{prop}

\begin{proof}
The inequality holds when $y^*=0$, so we need to deal with the case when $y^*\neq 0$.
Since both sides are homogeneous in all the arguments, we may rescale to assume that $y^{*} \cdot v = y^{*} \cdot x$.  Then we need to show that $s D(v) \cdot x \geq f(v)$.  If not, then
\begin{equation*}
f(v) - sD(v) \cdot x > 0,
\end{equation*}
so that $v-x \in \mathcal{C}^{\circ}$ by the Morse-type inequality.  But then we conclude that $y^{*} \cdot v > y^{*} \cdot x$, a contradiction.
\end{proof}

\begin{thrm} \label{general morse}
Let $f \in \HConc_{s}(\mathcal{C})$ be $+$-differentiable and satisfy a Morse-type inequality. Then for any $v \in \mathcal{C}_f$ and $y^*\in \mathcal{C}^{*}$ satisfying
\begin{align*}
\mathcal{H}f(D(v))-s v\cdot y^*>0,
\end{align*}
we have $D(v)-y^*\in \mathcal{C}^{*\circ}$. In particular, we have $D(v)-y^*\in \mathcal{C}^*_{\mathcal{H}f}$ and
\begin{align*}
\mathcal{H}f(D(v)-y^*)^{s-1/s}&\geq (\mathcal{H}f(D(v))-s v\cdot y^*)\mathcal{H}f(D(v))^{-1/s}\\
&=(f(v)-s v\cdot y^*)f(v)^{-1/s}.
\end{align*}
As a consequence, we get
\begin{align*}
\mathcal{H}f(D(v)-y^*)\geq f(v)-\frac{s^2}{s-1}v\cdot y^*.
\end{align*}
\end{thrm}

\begin{proof}
Note that $\mathcal{H}f(D(v))=f(v)$. First we claim that the inequality $f(v)-sv\cdot y^*>0$ implies $D(v)-y^* \in \mathcal{C}^{*\circ}$. To this end, fix some sufficiently small $y'^{*} \in \mathcal{C}^{*\circ}$ such that $y^*+y'^*$ still satisfies $f(v)-sv\cdot (y^*+y'^*)>0$.

Then by the ``reverse" Khovanskii-Teissier inequality, for some $\delta>0$ and any $x\in \mathcal{C}$ we have
\begin{align*}
D(v)\cdot x\geq \left(\frac{f(v)}{s(y^* + y'^*)\cdot v}\right) (y^*+y'^*) \cdot  x \geq (1+\delta)(y^*+y'^*)\cdot  x.
\end{align*}
This implies $D(v)-y^* \in \mathcal{C}^{*\circ}$.

By the definition of $\mathcal{H}f$ we have
\begin{align*}
 \mathcal{H}f(D(v)-y^*)&=\inf_{x\in \mathcal{C}^\circ}\left( \frac{(D(v)-y^*)\cdot x}{f(x)^{1/s}}\right)^{s/s-1}\\
 &\geq \left(\frac{f(v)-sy^*\cdot v}{f(v)}\right)^{s/s-1}\inf_{x\in \mathcal{C}^\circ}\left( \frac{D(v)\cdot x}{f(x)^{1/s}}\right)^{s/s-1}\\
 &=\mathcal{H}f(D(v))\left(\frac{f(v)-sy^*\cdot v}{f(v)}\right)^{s/s-1},
\end{align*}
where the second line follows from ``reverse" Khovanskii-Teissier inequality. To obtain the desired inequality, we only need to use the equality $\mathcal{H}f(D(v))=f(v)$ again.

To show the last inequality, we only need to note that the function $(1-x)^\alpha$ is convex for $x\in [0,1)$ if $\alpha \geq 1$. This implies $(1-x)^\alpha \geq 1-\alpha x$. Applying this inequality in our situation, we get
\begin{align*}
\mathcal{H}f(D(v)-y^*)&\geq \left( 1- \frac{s v\cdot y^*}{f(v)}\right)^{s/s-1} f(v)\\
& \geq f(v)-\frac{s^2}{s-1}v\cdot y^*.
\end{align*}
\end{proof}

\subsection{Boundary conditions}

Under certain conditions we can control the behaviour of
$\mathcal{H}f$ near the boundary, and thus obtain continuity.

\begin{defn}
Let $f \in \HConc_{s}(\mathcal{C})$ and let $\alpha \in (0,1)$.  We say that $f$ satisfies the sublinear boundary condition of order $\alpha$ if for any non-zero $v$ on the boundary of $\mathcal{C}$ and for any $x$ in the interior of $\mathcal{C}$, there exists a constant $C:=C(v, x)>0$ such that $f(v+\epsilon x)^{1/s}\geq C\epsilon^\alpha$.
\end{defn}

Note that the condition is always satisfied at $v$ if $f(v)>0$.  Furthermore, the condition is satisfied for any $v,x$ with $\alpha = 1$ by homogeneity and log-concavity, so the crucial question is whether we can decrease $\alpha$ slightly.

Using this sublinear condition, we get the vanishing of $\mathcal{H}f$ along the boundary.

\begin{prop} \label{sublinearcontinuity}
Let $f \in \HConc_{s}(\mathcal{C})$ satisfy the sublinear boundary condition of order $\alpha$. Then $\mathcal{H}f$ vanishes along the boundary. As a consequence, $\mathcal{H}f$ extends to a continuous function over $V^*$ by setting
$\mathcal{H}f=0$ outside $\mathcal{C}^*$.
\end{prop}

\begin{proof}
Let $w^*$ be a boundary point of $\mathcal{C}^*$.  Then there exists some non-zero $v\in \mathcal{C}$ such that $w^*\cdot v=0$. Fix $x\in \mathcal{C}^\circ$. By the definition of $\mathcal{H}f$ we get
\begin{align*}
\mathcal{H}f(w^*)^{s-1/s}\leq \frac{w^*\cdot (v+\epsilon x)}{f^{1/s}(v+\epsilon x)}\leq \frac{\epsilon w^*\cdot x}{C\epsilon^\alpha}.
\end{align*}
Letting $\epsilon$ tend to zero, we see $\mathcal{H}f(w^*)=0$.

To show the continuity, by Lemma \ref{uppersemilimit} we only need to verify
\begin{align*}
\lim _{\epsilon \rightarrow 0}\mathcal{H}f(w^* + \epsilon y^*)=0
\end{align*}
for some $y^* \in \mathcal{C}^{*\circ}$ (as any other limiting sequence is dominated by such a sequence).  This follows easily from
\begin{align*}
\mathcal{H}f(w^* + \epsilon y^*)^{s-1/s}&\leq \frac{(w^*+\epsilon y^*)\cdot (v+\epsilon x)}{f^{1/s}(v+\epsilon x)}\\
&\leq \frac{\epsilon(y^*\cdot v + w^*\cdot x + \epsilon y^*\cdot x)}{C \epsilon^\alpha}.
\end{align*}
\end{proof}

\begin{rmk}
If $f$ satisfies the sublinear condition, then $\mathcal{C}_{\mathcal{H}f} ^* =\mathcal{C}^{*\circ}$. This makes the statements of the previous results very clean. In the following section, the function $\widehat{\vol}$ 
has this nice property.
\end{rmk}

\section{Positivity for curves}
\label{section zariski}
We now study the basic properties of $\widehat{\vol}$ and of the Zariski decompositions for curves.  Some aspects of the theory will follow immediately from the formal theory of Section \ref{formal zariski section}; others will require a direct geometric argument.

We first outline how to apply the results of Section \ref{formal zariski section}.  Recall that $\widehat{\vol}$ is the polar transform of the volume function for divisors restricted to the nef cone. More precisely, we are now in the situation:
\begin{align*}
\mathcal{C}=\Nef^1 (X), \quad f=\vol,\quad \mathcal{C}^*=\Eff_1 (X),\quad \mathcal{H}f= \widehat{\vol}.
\end{align*}

Thus, to understand the properties of $\widehat{\vol}$ we need to recall the basic features of the volume function on the nef cone of divisors.
It is an elementary fact that the volume function on the nef cone of divisors is differentiable everywhere (with $D(A) = A^{n-1}$).  In the notation of Section \ref{legendresection} the cone $\Nef^{1}(X)_{\vol}$ coincides with the big and nef cone.  The Khovanskii-Teissier inequality (with Teissier proportionality) holds on the big and nef cone as recalled in Section \ref{section preliminaries}.  Finally, the volume for nef divisors satisfies the sublinear boundary condition of order $n-1/n$: this follows from an elementary intersection calculation using the fact that $N \cdot A^{n-1} \neq 0$ for any non-zero nef divisor $N$ and ample divisor $A$.

\begin{rmk}
\label{rmk decomposition kahler}
Due to the outline above, the proofs in this section depend only upon elementary facts about intersection theory, the Khovanskii-Teissier inequality and Teissier's proportionality theorem.  As discussed in the preliminaries, the arguments in this section thus extend immediately to smooth varieties over an arbitrary algebraically closed field and to the K\"ahler setting.
\end{rmk}

\subsection{Basic properties}

The following theorems collect the various analytic consequences for $\widehat{\vol}$.

\begin{thrm} \label{volcurvesbasicprops}
Let $X$ be a projective variety of dimension $n$.  Then:
\begin{enumerate}
\item $\widehat{\vol}$ is continuous and homogeneous of weight $n/n-1$ on $\Eff_{1}(X)$ and is positive precisely for the big classes.
\item For any big and nef divisor class $A$, we have $\widehat{\vol}(A^{n-1}) = \vol(A)$.
\item For any big curve class $\alpha$, there is a big and nef divisor class $B$ such that
\begin{equation*}
\widehat{\vol}(\alpha) = \left( \frac{B \cdot \alpha}{\vol(B)^{1/n}} \right)^{n/n-1}.
\end{equation*}
We say that the class $B$ computes $\widehat{\vol}(\alpha)$.
\end{enumerate}
\end{thrm}

The first two were already proved in \cite[Theorem 3.1]{xiao15}.

\begin{proof}
(1) follows immediately from Propositions \ref{prop H involution} and \ref{sublinearcontinuity}.  Since $D(A) = A^{n-1}$, (2) follows from the computation $$\widehat{\vol}(A^{n-1}) = D(A) \cdot A = A^{n}.$$
The existence in (3) follows from Theorem \ref{thm abstract zariski}.
\end{proof}

We also note the following easy basic linearity property, which follows immediately from the Khovanskii-Teissier inequalities.

\begin{thrm} \label{zardecomlinearity}
Let $X$ be a projective variety of dimension $n$ and let $\alpha$ be a big curve class.  If $A$ computes $\widehat{\vol}(\alpha)$, it also computes $\widehat{\vol}(c_{1}\alpha + c_{2}A^{n-1})$ for any positive constants $c_{1}$ and $c_{2}$.
\end{thrm}

After constructing Zariski decompositions below, we will see that in fact we can choose a possibly negative $c_{2}$ so long as $c_{1} \alpha + c_{2} A^{n-1}$ is a big class.

\subsection{Zariski decompositions for curves}

The following theorem is the basic result establishing the existence of Zariski decompositions for curve classes.

\begin{thrm} \label{thrm:existenceofzardecom}
Let $X$ be a projective variety of dimension $n$.  Any big curve class $\alpha$ admits a unique Zariski decomposition: there is a unique pair consisting of a big and nef divisor class $B_\alpha$ and a pseudo-effective curve class $\gamma$ satisfying $B_\alpha \cdot \gamma = 0$ and
\begin{equation*}
\alpha = B_\alpha^{n-1} + \gamma.
\end{equation*}
In fact $\widehat{\vol}(\alpha) = \widehat{\vol}(B_{\alpha}^{n-1}) = \vol(B_{\alpha})$.  In particular $B_{\alpha}$ computes $\widehat{\vol}(\alpha)$, and any big and nef divisor computing $\widehat{\vol}(\alpha)$ is proportional to $B_{\alpha}$.
\end{thrm}

\begin{proof}
The existence of the Zariski decomposition and the uniqueness of the positive part $B_{\alpha}^{n-1}$ follow from Theorem \ref{strong zariski equivalence}.  The uniqueness of $B_\alpha$ follows from Teissier proportionality for big and nef divisor classes.  It is clear that $B_{\alpha}$ computes $\widehat{\vol}(\alpha)$ by Theorem \ref{strong zariski equivalence}.  The last claim follows from Teissier proportionality and the fact that $\alpha \succeq B_{\alpha}^{n-1}$.
\end{proof}

As discussed before, conceptually the Zariski decomposition $\alpha = B_\alpha^{n-1} + \gamma$ captures the failure of log concavity of $\widehat{\vol}$: the term $B_\alpha^{n-1}$ captures all the of the positivity encoded by $\widehat{\vol}$ and is positive in a very strong sense, while the negative part $\gamma$ lies on the boundary of the pseudo-effective cone.

\begin{exmple}
\label{example proj bundle}
Let $X$ be the projective bundle over $\mathbb{P}^{1}$ defined by $\mathcal{O} \oplus \mathcal{O} \oplus \mathcal{O}(-1)$.  There are two natural divisor classes on $X$: the class $f$ of the fibers of the projective bundle and the class $\xi$ of the sheaf $\mathcal{O}_{X/\mathbb{P}^{1}}(1)$.  Using for example \cite[Theorem 1.1]{fulger11} and \cite[Proposition 7.1]{fl14}, one sees that $f$ and $\xi$ generate the algebraic cohomology classes with the relations $f^{2} = 0$, $\xi^{2}f = -\xi^{3} = 1$ and
\begin{equation*}
\Eff^{1}(X) = \Mov^{1}(X) = \langle f, \xi \rangle \qquad \qquad \Nef^{1}(X) = \langle f, \xi + f \rangle
\end{equation*}
and
\begin{align*}
\Eff_{1}(X) = \langle \xi f, \xi^{2} \rangle \qquad & \qquad \Nef_{1}(X) = \langle \xi f, \xi^{2} + \xi f \rangle  \\
 \CI_{1}(X) = \langle \xi f, & \xi^{2} + 2 \xi f \rangle.
\end{align*}
Using this explicit computation of the nef cone of the divisors, we have
\begin{equation*}
\widehat{\vol}(x \xi f + y \xi^{2}) = \inf_{a,b \geq 0} \frac{a y + bx}{(3ab^{2}+2b^{3})^{1/3}}
\end{equation*}
This is essentially a one-variable minimization problem due to the homogeneity in $a,b$.  It is straightforward to compute directly that for non-negative values of $x,y$:
\begin{align*}
\widehat{\vol}(x \xi f + y \xi^{2}) & = \left( \frac{3}{2}x - y \right) y^{1/2} \qquad  \textrm{ if }x \geq 2y; \\
& = \frac{x^{3/2}}{2^{1/2}} \qquad \qquad \qquad  \textrm{ if }x < 2y.
\end{align*}
Note that when $x<2y$, the class $x \xi f + y \xi^{2}$ no longer lies in the complete intersection cone -- to obtain $\widehat{\vol}$, Theorem \ref{thrm:existenceofzardecom} indicates that we must project $\alpha$ onto the complete intersection cone in the $y$-direction.  This exactly coheres with the calculation above.
\end{exmple}

The Zariski decomposition for curves is continuous.

\begin{thrm} \label{zardecomcontinuous}
Let $X$ be a projective variety of dimension $n$.  The function sending a big curve class $\alpha$ to its positive part $B_{\alpha}^{n-1}$ or to the corresponding divisor $B_{\alpha}$ is continuous.
\end{thrm}

\begin{proof}
The first statement follows from Theorem \ref{thrm positive part continuity}. The second then follows from the continuity of the inverse map to the $n-1$-power map.
\end{proof}

It is interesting to study whether the Zariski projection taking $\alpha$ to its positive part is $\mathcal{C}^1$. This is true on the ample cone -- the map $\Phi$ sending an ample divisor class $A$ to $A^{n-1}$ is a $\mathcal{C}^1$ diffeomorphism by the argument in Remark \ref{rmk CI}.

\begin{rmk}
The continuity of the Zariski decomposition does not extend to the entire pseudo-effective cone, even for surfaces.  For example, suppose that a surface $S$ admits a nef class $N$ which is a limit of (rescalings of) irreducible curve classes which each have negative self-intersection.  (A well-known example of such a surface is $\mathbb{P}^{2}$ blown up at 9 general points.)  For any $c \in [0,1]$ one can find a sequence of big divisors $\{ L_{i} \}$ whose limit is $N$ but whose positive parts have limit $c N$.
\end{rmk}

An important feature of the $\sigma$-decomposition for divisors is its concavity: given two big divisors $L_{1}, L_{2}$ we have
$$P_{\sigma}(L_{1}+L_{2}) \succeq P_{\sigma}(L_{1}) + P_{\sigma}(L_{2}).$$
However, the analogous property fails for curves:

\begin{exmple}
Let $X$ be a smooth projective variety such that $\CI_{1}(X)$ is not convex.  (An explicit example is given in Appendix.)  Then there are complete intersection classes $\alpha = B_{\alpha}^{n-1}$ and $\beta = B_{\beta}^{n-1}$ such that $\alpha+\beta$ is not a complete intersection class.  Let $B_{\alpha + \beta}^{n-1}$ denote the positive part of the Zariski decomposition for $\alpha+\beta$.  Then
\begin{equation*}
B_{\alpha+\beta}^{n-1} \preceq B_{\alpha}^{n-1} + B_{\beta}^{n-1}.
\end{equation*}
Furthermore, we can not have equality since the sum is not a complete intersection class.  Thus
\begin{equation*}
B_{\alpha+\beta}^{n-1} \precneqq B_{\alpha}^{n-1} + B_{\beta}^{n-1}.
\end{equation*}
\end{exmple}

However, one can still ask:

\begin{ques}
Fix $\alpha \in \Eff_{1}(X)$.  Is there a fixed class $\xi \in \CI_{1}(X)$ such that for any $\epsilon > 0$ there is a $\delta > 0$ satisfying
\begin{equation*}
B_{\alpha+\delta \beta}^{n-1} \preceq B_{\alpha + \epsilon \xi}^{n-1}
\end{equation*}
for every $\beta \in N_{1}(X)$ of bounded norm?
\end{ques}

This question is crucial for making sense of the Zariski decomposition of a curve class on the boundary of $\Eff_{1}(X)$ via taking a limit.

\subsection{Strict log concavity}

The following theorem is an immediate consequence of Theorem \ref{strong zariski equivalence}, which gives the strict log concavity of $\widehat{\vol}$.

\begin{thrm} \label{volcurvesconcavity}
Let $X$ be a projective variety of dimension $n$.  For any two pseudo-effective curve classes $\alpha, \beta$ we have
\begin{equation*}
\widehat{\vol}(\alpha + \beta)^{\frac{n-1}{n}} \geq \widehat{\vol}(\alpha)^{\frac{n-1}{n}} + \widehat{\vol}(\beta)^{\frac{n-1}{n}}.
\end{equation*}
Furthermore, if $\alpha$ and $\beta$ are big, then we obtain an equality if and only if the positive parts of $\alpha$ and $\beta$ are proportional.
\end{thrm}

\begin{proof}
The inequality is clear.  Combining the $+$-differentiability of $\vol$ with Theorem \ref{strong zariski equivalence}, we see the forward implication in the last sentence.  Conversely, if $\alpha$ and $\beta$ have proportional positive parts, then working directly from the definition it is clear that the sum of the positive parts is the (unique) positive part of $\alpha + \beta$, and the conclusion follows.
\end{proof}

\subsection{Differentiability}
\label{section derivative}

In \cite{bfj09} the derivative of the volume function was calculated using the positive product: given a big divisor class $L$ and any divisor class $E$, we have
\begin{equation*}
\left. \frac{d}{dt} \right|_{t=0} \vol(L + t E) = n \langle L^{n-1} \rangle \cdot E.
\end{equation*}
In this section we prove an analogous statement for curve classes.  For curves, the big and nef divisor class $B$ occurring in the Zariski decomposition plays the role of the positive product, and the homogeneity constant $n/n-1$ plays the role of $n$.

\begin{thrm} \label{thrm:derivative}
Let $X$ be a projective variety of dimension $n$, and let $\alpha$ be a big curve class with Zariski decomposition $\alpha = B^{n-1} + \gamma$.  Let $\beta$ be any curve class.  Then $\widehat{\vol}(\alpha + t \beta)$ is differentiable at $0$ and
\begin{align*}
\left. \frac{d}{dt} \right|_{t=0} \widehat{\vol}(\alpha + t \beta) & = \frac{n}{n-1} B \cdot \beta.
\end{align*}
In particular, the function $\widehat{\vol}$ is $\mathcal{C}^{1}$ on the big cone of curves.  If $C$ is an irreducible curve on $X$, then we can instead write
\begin{align*}
\left. \frac{d}{dt} \right|_{t=0} \widehat{\vol}(\alpha + t C)= \frac{n}{n-1} \vol(B|_{C}).
\end{align*}
\end{thrm}

\begin{proof}
This follows immediately from Propositions \ref{diffuniqueness} and \ref{derivativeofdual} since $G_{\alpha}\cup \{0\}$ consists of a single ray by the last statement of Theorem \ref{thrm:existenceofzardecom}.
\end{proof}

\begin{exmple}
\label{example derivative}
We return to the setting of Example \ref{example proj bundle}: let $X$ be the projective bundle over $\mathbb{P}^{1}$ defined by $\mathcal{O} \oplus \mathcal{O} \oplus \mathcal{O}(-1)$.  Using our earlier notation we have
\begin{align*}
\Eff_{1}(X) = \langle \xi f, \xi^{2} \rangle \qquad
\end{align*}
and
\begin{align*}
\widehat{\vol}(x \xi f + y \xi^{2}) & = \left( \frac{3}{2}x - y \right) y^{1/2} \qquad  \textrm{ if }x \geq 2y; \\
& = \frac{x^{3/2}}{2^{1/2}} \qquad \qquad \qquad  \textrm{ if }x < 2y.
\end{align*}
We focus on the complete intersection region where $x \geq 2y$.  Then we have
\begin{equation*}
x \xi f + y \xi^{2} = \left(  \frac{x - 2y}{2 y^{1/2}} f + y^{1/2} (\xi + f) \right)^{2}.
\end{equation*}
The divisor in the parentheses on the right hand side is exactly the $B$ appearing in the Zariski decomposition expression for $x \xi f + y \xi^{2}$.  Thus, we can calculate the directional derivative of $\widehat{\vol}$ along a curve class $\beta$ by intersecting against this divisor.

For a very concrete example, set $\alpha = 3 \xi f + \xi^{2}$, and consider the behavior of $\widehat{\vol}$ for
\begin{equation*}
\alpha_{t} := 3 \xi f + \xi^{2} - t(2\xi f+ \xi^{2}).
\end{equation*}
Note that $\alpha_{t}$ is pseudo-effective precisely for $t \leq 1$.  In this range, the explicit expression for the volume above yields
\begin{align*}
\widehat{\vol}(\alpha_{t}) & = \left( \frac{7}{2} - 2t \right) (1-t)^{1/2}, \\
\frac{d}{dt} \widehat{\vol}(\alpha_{t}) 
& = -3(1-t)^{1/2} - \frac{3}{4}(1-t)^{-1/2}.
\end{align*}
Note that this calculation agrees with the prediction of Theorem \ref{thrm:derivative}, which states that if $B_{t}$ is the divisor defining the positive part of $\alpha_{t}$ then
\begin{align*}
\frac{d}{dt} \widehat{\vol}(\alpha_{t}) & = \frac{3}{2} B_{t} \cdot (2 \xi f + \xi^{2}) \\
& = \frac{-3}{2}   \left(  \frac{(3-2t) - 2(1-t)}{2 (1-t)^{1/2}} + 2(1-t)^{1/2} \right).
\end{align*}
In particular, the derivative decreases to $-\infty$ as $t$ approaches $1$ (and the coefficients of the divisor $B$ also increase without bound).  This is a surprising contrast to the situation for divisors.  Note also that $\widehat{\vol}$ is not convex on this line segment, while $\vol$ is convex in any pseudo-effective direction in the nef cone of divisors by the Morse inequality.
\end{exmple}

\subsection{Negative parts}

We next analyze the structure of the negative part of the Zariski decomposition.  First we have:

\begin{lem} \label{negnotmovable}
Let $X$ be a projective variety.  Suppose $\alpha$ is a big curve class and write $\alpha = B^{n-1} + \gamma$ for its Zariski decomposition.  If $\gamma \neq 0$ then $\gamma \not \in \Mov_{1}(X)$.
\end{lem}

\begin{proof}
Since $B$ is big and $B \cdot \gamma = 0$, $\gamma$ cannot be movable if it is non-zero.
\end{proof}

For the Zariski decomposition under $\widehat{\vol}$, we can not guarantee the negative part is the class of an effective curve.  As in \cite{fl14}, it is more reasonable to ask if the negative part is the pushforward of a pseudo-effective class from a proper subvariety.  Note that this property is automatic when the negative part is represented by an effective class, and for surfaces it is actually equivalent to asking that the negative part be effective.  In general this subtle property of pseudo-effective classes is crucial for inductive arguments on dimension.

\begin{prop} \label{negpartrigid}
Let $X$ be a projective variety of dimension $n$.  Let $\alpha$ be a big curve class and write $\alpha = B^{n-1} + \gamma$ for its Zariski decomposition.  There is a proper subscheme $i: V \subsetneq X$ and a pseudo-effective class $\gamma' \in N_{1}(V)$ such that $i_{*}\gamma' = \gamma$.
\end{prop}

\begin{proof}
We may choose an effective nef $\mathbb{R}$-Cartier divisor $D$ whose class is $B$.  By resolving the base locus of a sufficiently high multiple of $D$ we obtain a blow-up $\phi: Y \to X$, a birational morphism $\psi: Y \to Z$ and an effective ample divisor $A$ on $Z$ such that after replacing $D$ by some numerically equivalent divisor we have $\phi^{*}D \geq \psi^{*}A$.  Write $E$ for the difference of these two divisors and set $V_{Y}$ to be the union of $\Supp(E)$ with the $\psi$-exceptional locus.

There is a pseudo-effective curve class $\gamma_{Y}$ on $Y$ which pushes forward to $\gamma$ and thus satisfies $\phi^{*}D \cdot \gamma_{Y} = 0$.  There is an infinite sequence of effective $1$-cycles $C_{i}$ such that $\lim_{i \to \infty} [C_{i}] = \gamma_{Y}$.  Each effective cycle $C_{i}$ can be decomposed as a sum $C_{i} = T_{i} + T_{i}'$ where $T_{i}'$ consists of the components contained in $V_{Y}$ and $T_{i}$ consists of the rest.

Note that
\begin{equation*}
\lim_{i \to \infty} A \cdot \psi_{*}T_{i} \leq \lim_{i \to \infty} \phi^{*}D \cdot T_{i} = 0.
\end{equation*}
This shows that $\lim_{i \to \infty} [T_{i}]$ converges to a pseudo-effective curve class $\beta \in N_{1}(Y)$ satisfying $\psi_{*}\beta = 0$.

Clearly $\lim_{i \to \infty}[T_{i}']$ is the pushforward of a pseudo-effective curve class from $V_{Y}$.  \cite[Theorem 4.1]{djv13} (which holds in the singular case by the same argument) shows that $\beta$ is also the pushforward of a pseudo-effective curve class on $V_{Y}$.  Thus $\gamma_{Y}$ is the pushforward of a pseudo-effective curve class on $V_{Y}$.  Pushing forward to $X$, we see that $\gamma$ is the pushforward of a pseudo-effective curve class on $V := \phi(V_{Y})$.  Note that $V$ is a proper subset of $X$ since $\phi$ is birational.
\end{proof}

\begin{rmk}
In contrast, for the Zariski decomposition of curves in the sense of Boucksom
(see \cite[Theorem 3.3 and Lemma 3.5]{xiao15}) the negative part can always be represented by an effective curve
which is very rigidly embedded in $X$.  This has a similar feel as the $\sigma$-decomposition of \cite{Nak04} for curve classes.
\end{rmk}

\subsection{Birational behavior}
\label{birational subsection}

We next use the Zariski decomposition to analyze the behavior of positivity of curves under birational maps $\phi: Y \to X$.  Note that (in contrast to divisors) the birational pullback can only decrease the positivity for curve classes: we have
\begin{equation*}
\widehat{\vol}(\alpha) \geq \widehat{\vol}(\phi^{*}\alpha).
\end{equation*}
In fact pulling back does not preserve pseudo-effectiveness, and even for a movable class we can have a strict inequality of $\widehat{\vol}$ (for example, a big movable class can pull back to a movable class on the pseudo-effective boundary).  Again guided by \cite{fl14}, the right approach is to consider all $\phi_{*}$-preimages of $\alpha$ at once.

\begin{prop} \label{birvolstatement}
Let $\phi: Y \to X$ be a birational morphism of projective varieties of dimension $n$.  Let $\alpha$ be a big curve class on $X$ with Zariski decomposition $B^{n-1} + \gamma$.  Let $\mathcal{A}$  be the set of all pseudo-effective curve classes $\alpha'$ on $Y$ satisfying $\phi_{*}\alpha' = \alpha$.  Then
\begin{equation*}
\sup_{\alpha' \in \mathcal{A}} \widehat{\vol}(\alpha') = \widehat{\vol}(\alpha).
\end{equation*}
This supremum is achieved by an element $\alpha_{Y} \in \mathcal{A}$.
\end{prop}

\begin{proof}
Suppose $\alpha' \in \mathcal{A}$.  Since $\phi_{*}\alpha' = \alpha$, it is clear from the projection formula that $\widehat{\vol}(\alpha') \leq \widehat{\vol}(\alpha)$.  Conversely, set $\gamma_{Y}$ to be any pseudo-effective curve class on $Y$ pushing forward to $\gamma$.  Define $\alpha_{Y} = \phi^{*}B^{n-1} + \gamma_{Y}$.  Since $\phi^{*}B \cdot \gamma_{Y} = 0$, by Theorem \ref{thrm:existenceofzardecom} this expression is the Zariski decomposition for $\alpha_{Y}$.  In particular $\widehat{\vol}(\alpha_{Y}) = \widehat{\vol}(\alpha)$.
\end{proof}

This proposition indicates the existence of some ``distinguished'' preimages of $\alpha$ with maximum $\widehat{\vol}$.  In fact, these distinguished preimages also have a very nice structure.

\begin{prop} \label{birpospartstructure}
Let $\phi: Y \to X$ be a birational morphism of projective varieties of dimension $n$.  Let $\alpha$ be a big curve class on $X$ with Zariski decomposition $B^{n-1} + \gamma$.  Set $\mathcal{A}'$ to be the set of all pseudo-effective curve class $\alpha'$ on $Y$ satisfying $\phi_{*}\alpha' = \alpha$ and $\widehat{\vol}(\alpha') = \widehat{\vol}(\alpha)$.  Then
\begin{enumerate}
\item Every $\alpha' \in \mathcal{A}'$ has a Zariski decomposition of the form
\begin{equation*}
\alpha' = \phi^{*}B^{n-1} + \gamma'.
\end{equation*}
Thus $\mathcal{A}' = \{ \phi^{*}B^{n-1} + \gamma' \, | \, \gamma' \in \Eff_{1}(Y), \phi_{*}\gamma' = \gamma \}$ is determined by the set of pseudo-effective preimages of $\gamma$.
\item These Zariski decompositions are stable under adding $\phi$-exceptional curves: if $\xi$ is a pseudo-effective curve class satisfying $\phi_{*}\xi = 0$, then for any $\alpha' \in \mathcal{A}'$ we have
\begin{equation*}
\alpha' + \xi = \phi^{*}B^{n-1} + (\gamma' + \xi)
\end{equation*} is the Zariski decomposition for $\alpha' + \xi$.
\end{enumerate}
\end{prop}

\begin{proof}
To see (1), note that
\begin{equation*}
\frac{\phi^{*}B}{\vol(B)^{1/n}} \cdot \alpha' = \frac{B}{\vol(B)^{1/n}} \cdot \alpha = \widehat{\vol}(\alpha).
\end{equation*}
Thus if $\widehat{\vol}(\alpha') = \widehat{\vol}(\alpha)$ then $\widehat{\vol}(\alpha')$ is computed by $\phi^{*}B$.  By Theorem \ref{thrm:existenceofzardecom} we obtain the statement.

(2) follows immediately from (1), since
\begin{equation*}
\widehat{\vol}(\alpha) = \widehat{\vol}(\alpha') \leq \widehat{\vol}(\alpha' + \xi) \leq \widehat{\vol}(\alpha)
\end{equation*}
by Proposition \ref{birvolstatement}.
\end{proof}

While there is not necessarily a \emph{uniquely} distinguished $\phi_{*}$-preimage of $\alpha$, there \emph{is} a uniquely distinguished complete intersection class on $Y$ whose $\phi$-pushforward lies beneath $\alpha$ -- namely, the positive part of any sufficiently large class pushing forward to $\alpha$.  This is the analogue in our setting of the ``movable transform'' of \cite{fl14}.

\subsection{Morse-type inequality for curves}
\label{section morse ineq}

In this section we prove a Morse-type inequality for curves under the volume function $\widehat{\vol}$.  First let us recall the algebraic Morse inequality for nef divisor classes over smooth projective varieties. If $A, B$ are nef divisor classes on a smooth projective variety $X$ of dimension $n$, then by \cite[Example 2.2.33]{lazarsfeld04} (see also \cite{Dem85morse}, \cite{Siu93matsusaka}, \cite{Tra95morse})
$$
\vol(A-B)\geq A^n - n A^{n-1}\cdot B.
$$
In particular, if $A^n - n A^{n-1}\cdot B>0$, then $A-B$ is big. This gives us a very useful bigness criterion for the difference of two nef divisors.

By analogy with the divisor case, we can ask:
\begin{itemize}
  \item Let $X$ be a projective variety of dimension $n$, and let $\alpha, \gamma\in \Eff_1(X)$ be two nef (movable)
       curve classes. Is there a criterion for the bigness of $\alpha-\gamma\in  \Eff_1(X)$ using only intersection numbers defined by $\alpha, \gamma$?
\end{itemize}

Inspired by \cite{Xia13}, we give such a criterion using the $\widehat{\vol}$ function. In \cite{lehmannxiao2015b}, we answer the above question by giving a slightly different criterion which needs the refined structure of the movable cone of curves.  The following results follow from Theorem \ref{general morse}.

\begin{thrm}
\label{thm curve morseinequality}
Let $X$ be a projective variety of dimension $n$.  Let $\alpha$ be a big curve class and let $\beta$ be a movable curve class.  Write $\alpha = B^{n-1} + \gamma$ for the Zariski decomposition of $\alpha$. Then
\begin{align*}
\widehat{\vol}(\alpha - \beta)^{n-1/n}
&\geq (\widehat{\vol}(\alpha)-n B \cdot \beta)\cdot \widehat{\vol}(\alpha)^{-1/n}\\
&=(B^{n} - n B \cdot \beta)\cdot (B^{n})^{-1/n}.
\end{align*}
In particular, we have
\begin{equation*}
\widehat{\vol}(\alpha - \beta) \geq B^{n} - \frac{n^2}{n-1} B \cdot \beta.
\end{equation*}
\end{thrm}

\begin{proof}
The theorem follows immediately from Theorem \ref{general morse} and the fact that $\alpha \succeq B^{n-1}$.
\end{proof}

\begin{cor}
\label{cor curve morseinequality}
Let $X$ be a projective variety of dimension $n$.  Let $\alpha$ be a big curve class and let $\beta$ be a movable curve class.  Write $\alpha = B^{n-1} + \gamma$ for the Zariski decomposition of $\alpha$.  If
\begin{equation*}
\widehat{\vol}(\alpha) - nB \cdot \beta > 0
\end{equation*}
then $\alpha - \beta$ is big.
\end{cor}

\begin{rmk}
Superficially, the above theorem appears to differ from the classical algebraic Morse inequality for nef divisors, since $\alpha$ can be any big curve class.  However, using the Zariski decomposition one sees that the statement for $\alpha$ is essentially equivalent to the statement for the positive part of $\alpha$, so that Theorem \ref{thm curve morseinequality} is really a claim about nef curve classes.
\end{rmk}

\begin{exmple}
\label{exmple optimal n}
The constant $n$ is optimal in Corollary \ref{cor curve morseinequality}. Indeed, for any $\epsilon>0$ there exists a projective variety $X$ such that
$$
\widehat{\vol}(\alpha)-(n-\epsilon)B_\alpha \cdot \gamma >0,
$$
for some $\alpha\in \Eff_1(X)$ and $\gamma\in \Mov_1(X)$ but $\alpha-\gamma$ is not a big curve class.

To find such a variety, let $E$ be an elliptic curve with complex multiplication and set $X = E^{\times n}$. The pseudo-effective cone of divisors $\Eff^1(X)$ is identified with the cone of constant positive $(1,1)$-forms, while the pseudo-effective cone of curves $\Eff_1(X)$ is identified with the cone of constant positive $(n-1,n-1)$-forms.  Furthermore, every strictly positive $(n-1, n-1)$-form is a $(n-1)$-self-product of a strictly positive $(1,1)$-form.
We set
\begin{align*}
B_\alpha=i\sum_{j=1} ^n dz^j\wedge d\bar z ^j, \qquad B_\gamma=i\sum_{j=1} ^n \lambda_j dz^j\wedge d\bar z ^j.
\end{align*}
Here the $\lambda_j >0$.
Let $\alpha=B_\alpha ^{n-1}$ and $\gamma=B_\gamma ^{n-1}$.  Then $\widehat{\vol}(\alpha)-(n-\epsilon)B_\alpha \cdot \gamma >0$ is equivalent to
\begin{align*}
\sum_{j=1} ^n \lambda_1...\widehat{\lambda}_j...\lambda_n  <\frac{n}{n-\epsilon},
\end{align*}
and $\alpha-\gamma$ being big is equivalent to
\begin{align*}
\lambda_1...\widehat{\lambda}_j...\lambda_n<1
\end{align*}
for every $j$. Now it is easy to see we can always choose $\lambda_1,..., \lambda_n$ such that the first inequality holds but the second does not hold.
\end{exmple}

\begin{rmk}
Using the cone duality $\overline{\mathcal{K}}^* =\mathcal{N}$ and Theorem \ref{thrm appendix reserve KT} in Appendix A, it is easy to extend the above Morse-type inequality for curves to positive currents of bidimension $(1,1)$ over compact K\"ahler manifolds.
\end{rmk}

One wonders if Theorem \ref{thm curve morseinequality} can be improved:

\begin{ques} \label{morseques}
Let $X$ be a projective variety of dimension $n$.  Let $\alpha$ be a big curve class and let $\beta$ be a movable curve class.  Write $\alpha = B^{n-1} + \gamma$ for the Zariski decomposition of $\alpha$.  Is
\begin{equation*}
\widehat{\vol}(\alpha - \beta) \geq \vol(\alpha) - n B \cdot \beta?
\end{equation*}

\end{ques}

\begin{rmk}
\label{rmk conj morse}
By Theorem \ref{thm curve morseinequality}, if $\widehat{\vol}(\alpha) - n B \cdot \beta>0$ then $\widehat{\vol}$ is $\mathcal{C}^1$ at the point $\alpha-s\beta$ for every $s\in [0,1]$. The derivative formula of $\widehat{\vol}$ implies
\begin{align*}
\widehat{\vol} (\alpha-\beta)-\widehat{\vol} (\alpha)=\int_0 ^1
-\frac{n}{n-1}B_{\alpha-s\beta} \cdot \beta \, \, ds,
\end{align*}
where $B_{\alpha-s\beta}$ is the big and nef divisor class defining the Zariski decomposition of $\alpha-s\beta$.  To give an affirmative answer to Question \ref{morseques}, we conjecture the following:
\begin{align*}
B_{\alpha-s\beta} \cdot \beta \leq (n-1)B_{\alpha} \cdot \beta\ \textrm{for every}\ s\in [0,1].
\end{align*}
Without loss of generality, we can assume $B_{\alpha} \cdot \beta >0$. Then by continuity of the decomposition, this inequality holds for $s$ in a neighbourhood of $0$. At this moment, we do not know how to see this neighbourhood covers $[0,1]$.
\end{rmk}

\begin{rmk} \label{aposteriori}
In this section we have seen how to use abstract convex analysis to understand the derivative and geometric inequalities for the volume function for curves.  Dually, one could in theory use the same approach to understand properties of the volume function of divisors.

Note that since the volume function for divisors is $n$-concave we have $\mathcal{H}^{2}\vol = \vol$, so that we can apply the results of Sections \ref{legendresection} and \ref{formal zariski section} to $\vol$.  Once we know a few key properties for $\vol$ or $\mathcal{H}\vol$, such as differentiability, then many well-known results for divisors (the Brunn-Minkowski inequality, the existence of $\sigma$-decompositions, the explicit expression for the derivative, etc.) follow immediately from the general set-up.  The polar transform of the volume function is analyzed in more detail in \cite{lehmannxiao2015b}.
\end{rmk}

\section{Toric varieties}
\label{toric section}
In this section $X$ will denote a simplicial projective toric variety of dimension $n$.  In terms of notation, $X$ will be defined by a fan $\Sigma$ in a lattice $N$ with dual lattice $M$.  We let $\{ v_{i} \}$ denote the primitive generators of the rays of $\Sigma$ and $\{ D_{i} \}$ denote the corresponding classes of $T$-divisors.

\subsection{Mixed volumes}

Suppose that $L$ is a big movable divisor class on the toric variety $X$.  Then $L$ naturally defines a (non-lattice) polytope $Q_{L}$: if we choose an expression $L = \sum a_{i}D_{i}$, then
\begin{equation*}
Q_{L} = \{ u \in M_{\mathbb{R}} |  \langle u,v_{i} \rangle + a_{i} \geq 0\}
\end{equation*}
and changing the choice of representative corresponds to a translation of $Q_{L}$.
Conversely, suppose that $Q$ is a full-dimensional polytope such that the unit normals to the facets of $Q$ form a subset of the rays of $\Sigma$.  Then $Q$ uniquely determines a big movable divisor class $L_{Q}$ on $X$.  The divisors in the interior of the movable cone correspond to those polytopes whose facet normals coincide with the rays of $\Sigma$.

Given polytopes $Q_{1},\ldots,Q_{n}$, let $V(Q_{1},\ldots,Q_{n})$ denote the mixed volume of the polytopes.  \cite{bfj09} explains that the positive product of big movable divisors $L_{1},\ldots,L_{n}$ can be interpreted via the mixed volume of the corresponding polytopes:
\begin{equation*}
\langle L_{1} \cdot \ldots \cdot L_{n} \rangle = n! V(Q_{1},\ldots,Q_{n}).
\end{equation*}

Now suppose that $\alpha$ lies in the interior of $\Mov_{1}(X)$.  Using \cite[Theorem 1.8]{lehmannxiao2015b}, we see that $\alpha = \langle L^{n-1} \rangle$ for some big movable divisor class $L$.  Let $P_{\alpha}$ denote the polytope corresponding to $L$.  Reinterpreting $\langle L^{n-1} \rangle \cdot A$ as a positive product for an ample divisor $A$, we see that the volume is
\begin{equation*}
\inf_{Q} \left( \frac{n! V(P_{\alpha}^{n-1},Q)}{ n!^{1/n}\vol(Q)^{1/n}} \right)^{n/n-1} = n! \inf_{Q} \left( \frac{V(P_{\alpha}^{n-1},Q)}{\vol(Q)^{1/n}} \right)^{n/n-1}
\end{equation*}
where $Q$ varies over all polytopes whose normal fan is refined by $\Sigma$.

\subsection{Computing the Zariski decomposition}

The nef cone of divisors and pseudo-effective cone of curves on $X$ can be computed algorithmically.  Thus, for any face $F$ of the nef cone, by considering the $(n-1)$-product and adding on any curve classes in the dual face, one can easily divide $\Eff_{1}(X)$ into regions where the positive product is determined by a class on $F$.  In practice this is a good way to compute the Zariski decomposition (and hence the volume) of curve classes on $X$.

In the other direction, suppose we start with a big curve class $\alpha$.  On a toric variety, every big and nef divisor is semi-ample (that is, the pullback of an ample divisor on a toric birational model).  Thus, the Zariski decomposition is characterized by the existence of a birational toric morphism $\pi: X \to X'$ such that:
\begin{itemize}
\item the class $\pi_{*}\alpha \in N_{1}(X')$ coincides with $A^{n-1}$ for some ample divisor $A$, and
\item $\alpha - (\pi^*A)^{n-1}$ is pseudo-effective.
\end{itemize}
Thus one can compute the Zariski decomposition and volume for $\alpha$ by the following procedure.
\begin{enumerate}
\item For each toric birational morphism $\pi: X \to X'$, check whether $\pi_{*}\alpha$ is in the complete intersection cone.  If so, there is a unique big and nef divisor $A_{X'}$ such that $A_{X'}^{n-1} = \pi_{*}\alpha$.
\item Check if $\alpha - (\pi^{*}A_{X'})^{n-1}$ is pseudo-effective.
\end{enumerate}
The first step involves solving polynomial equations to deduce the equality of coefficients of numerical classes, but otherwise this procedure is completely algorithmic.
(Note that there may be no natural pullback from $\Eff_{1}(X')$ to $\Eff_{1}(X)$, and in particular, the calculation of $(\pi^*A_{X'})^{n-1}$ is not linear in $A_{X'}^{n-1}$.)

\begin{exmple}\label{example toric zariski}
Let $X$ be the toric variety defined by a fan in $N = \mathbb{Z}^{3}$ on the rays
\begin{align*}
v_{1} & = (1,0,0)  \qquad \qquad \, \, \, \,  v_{2} = (0,1,0)  \qquad \qquad  \, \, \, \,\, \, v_{3} = (1,1,1) \\
v_{4} & = (-1,0,0) \qquad \qquad v_{5} = (0,-1,0) \qquad \qquad  v_{6}  = (0,0,-1)
\end{align*}
with maximal cones
\begin{align*}
\langle v_{1}, v_{2}, v_{3} \rangle, \, \langle v_{1}, v_{2}, v_{6} \rangle, \, \langle v_{1}, v_{3}, v_{5} \rangle, \, \langle v_{1}, v_{5}, v_{6} \rangle, \\
\langle v_{2}, v_{3}, v_{4} \rangle, \, \langle v_{2}, v_{4}, v_{6} \rangle, \, \langle v_{3}, v_{4}, v_{5} \rangle, \, \langle v_{4}, v_{5}, v_{6} \rangle.
\end{align*}
The Picard rank of $X$ is $3$.  Letting $D_{i}$ and $C_{ij}$ be the divisors and curves corresponding to $v_{i}$ and $\overline{v_{i}v_{j}}$ respectively, we have intersection product
\begin{equation*}
\begin{array}{c|c|c|c}
& D_{1} & D_{2} & D_{3} \\ \hline
C_{12} & -1 & -1 & 1 \\ \hline
C_{13} & 0 & 1 & 0 \\ \hline
C_{23} & 1 &  0 & 0
\end{array}
\end{equation*}
Standard toric computations show that:
\begin{align*}
\Eff^{1}(X) = \langle D_1, D_2, D_3 \rangle \qquad & \qquad \Nef^{1}(X) = \langle D_1+D_3, D_2+D_3, D_3 \rangle \\
\Mov^1(X) = \langle D_1 &+D_2, D_1+D_3,D_2+D_3,D_3 \rangle
\end{align*}
and
\begin{align*}
\Eff_{1}(X) = \langle C_{12}, C_{13}, C_{23} \rangle \qquad & \qquad \Nef_{1}(X) = \langle C_{12}+C_{13}+C_{23}, C_{13},C_{23} \rangle.
\end{align*}
$X$ admits a unique flip and has only one birational contraction corresponding to the face of $\Nef^{1}(X)$ generated by $D_{1}+D_{3}$ and $D_{2}+D_{3}$.  Set $B_{a,b} = aD_{1}+bD_{2}+(a+b)D_{3}$.  The complete intersection cone is given by taking the convex hull of the boundary classes
\begin{equation*}
B_{a,b}^2 = T_{a,b} = 2abC_{12} + (a^{2}+2ab)C_{13} + (b^{2}+2ab)C_{23}
\end{equation*}
and the face of $\Nef_{1}(X)$ spanned by $C_{13},C_{23}$.

For any big class $\alpha$ not in $\CI_{1}(X)$, the positive part can be computed on the unique toric birational contraction $\pi: X \to X'$ given by contracting $C_{12}$.  In practice, the procedure above amounts to solving $\alpha - t C_{12} = T_{a,b}$ for some $a,b,t$.  If $\alpha = xC_{12} + yC_{13} + zC_{23}$, this yields the quadratic equation $4(y-x+t)(z-x+t) = (x-t)^{2}$.  Solving this for $t$ tells us $\gamma = tC_{12}$, and the volume can then easily be computed.
\end{exmple}

\section{Hyperk\"ahler manifolds}
\label{hyperkahler section}

Throughout this section $X$ will denote a hyperk\"ahler variety of dimension $n$ (with $n=2m$).  We will continue to work in the projective setting.  However, as explained in Section \ref{kahler background sec}, Demailly's conjecture on transcendental Morse inequality is known for hyperk\"ahler manifolds.  Thus all the results in this section and related results in \cite{lehmannxiao2015b} can extended accordingly in the K\"ahler setting for hyperk\"ahler varieties with no qualifications.

Let $\sigma$ be a symplectic holomorphic form on $X$.  For a real divisor class $D \in N^{1}(X)$ the Beauville-Bogomolov quadratic form is defined as
\begin{align*}
q(D)= D^2 \cdot \{(\sigma \wedge \bar \sigma)\}^{n/2-1},
\end{align*}
where we normalize the symplectic form $\sigma$ such that
\begin{align*}
q(D)^{n/2}= D^n.
\end{align*}

As proved in \cite[Section 4]{Bou04}, the bilinear form $q$ is compatible with the volume function and $\sigma$-decomposition for divisors in the following way:
\begin{enumerate}
\item The cone of movable divisors is $q$-dual to the pseudo-effective cone.
\item If $D$ is a movable divisor then $\vol(D) = q(D,D)^{n/2} = D^{n}$.
\item For a pseudo-effective divisor $D$ write $D = P_{\sigma}(D) + N_{\sigma}(D)$ for its $\sigma$-decomposition.  Then $q(P_{\sigma}(D),N_{\sigma}(D)) = 0$, and if $N_{\sigma}(D) \neq 0$ then $q(N_{\sigma}(D),N_{\sigma}(D))<0$.
\end{enumerate}
The bilinear form $q$ induces an isomorphism $\psi: N^{1}(X) \to N_{1}(X)$ by sending a divisor class $D$ to the curve class defining the linear function $q(D,-)$.  We obtain an induced bilinear form $q$ on $N_{1}(X)$ via the isomorphism $\psi$, so that for curve classes $\alpha$, $\beta$
\begin{equation*}
q(\alpha,\beta) = q(\psi^{-1}\alpha,\psi^{-1}\beta) = \psi^{-1}\alpha \cdot \beta.
\end{equation*}
In particular, two cones $\mathcal{C}, \mathcal{C}'$ in $N^{1}(X)$ are $q$-dual if and only if $\psi(\mathcal{C})$ is dual to $\mathcal{C}'$ under the intersection pairing (and similarly for cones of curves).  In this section we verify that the bilinear form $q$ on $N_{1}(X)$ is compatible with the volume and Zariski decomposition for curve classes in the same way as for divisors.

\begin{rmk}
Since the signature of the Beauville-Bogomolov form is $(1,\dim N^{1}(X)-1)$, one can use the Hodge inequality to analyze the Zariski decomposition as in Example \ref{bilinearexample}.  We will instead give a direct geometric argument to emphasize the ties with the divisor theory.
\end{rmk}

We first need the following proposition.

\begin{prop} \label{hyperkahlerposprod}
Let $D$ be a big movable divisor class on $X$.  Then we have 
\begin{equation*}
\psi(D) = \frac{\langle D^{n-1} \rangle}{\vol(D)^{n-2/n}}.
\end{equation*}
In particular, the complete intersection cone coincides with the $\psi$-image of the nef cone of divisors and if $A$ is a big and nef divisor then $\widehat{\vol}(\psi(A)) = \vol(A)^{1/n-1}$.
\end{prop}

\begin{proof}
First note that $\psi(D)$ is contained in $\Mov_{1}(X)$.  Indeed, since the movable cone of divisors is $q$-dual to the pseudo-effective cone of divisors by \cite[Proposition 4.4]{Bou04}, the $\psi$-image of the movable cone of divisors is dual to the pseudo-effective cone of divisors.

For any big movable divisor $L$, the basic equality for bilinear forms shows that
\begin{equation*}
L \cdot \psi(D) = q(L,D) = \frac{1}{2} (\vol(L+D)^{2/n} - \vol(L)^{2/n} - \vol(D)^{2/n}).
\end{equation*}
In \cite[Theorem 1.7]{lehmannxiao2015b} we show that $\vol(L+D)^{1/n} \geq \vol(L)^{1/n} + \vol(D)^{1/n}$ with equality if and only if $L$ and $D$ are proportional.  Squaring and rearranging, we see that
\begin{equation*}
\frac{L \cdot \psi(D)}{\vol(L)^{1/n}} \geq \vol(D)^{1/n}
\end{equation*}
with equality if and only if $L$ is proportional to $D$.  
By \cite[Proposition 3.3 and Theorem 3.12]{lehmannxiao2015b} we immediately get that
\begin{equation*}
\psi(D) = \frac{\langle D^{n-1} \rangle}{\vol(D)^{n-2/n}}.
\end{equation*}
The final statements follow immediately.
\end{proof}

\begin{thrm} \label{secondhyperkahlertheorem}
Let $q$ denote the Beauville-Bogomolov form on $N_{1}(X)$.  Then:
\begin{enumerate}
\item The complete intersection cone of curves is $q$-dual to the pseudo-effective cone of curves.
\item If $\alpha$ is a complete intersection curve class then $\widehat{\vol}(\alpha) = q(\alpha,\alpha)^{n/2(n-1)}$.
\item For a big class $\alpha$ write $\alpha = B^{n-1} + \gamma$ for its Zariski decomposition.  Then $q(B^{n-1},\gamma) = 0$ and if $\gamma$ is non-zero then $q(\gamma,\gamma) < 0$.
\end{enumerate}
\end{thrm}

\begin{proof}
For (1), since the complete intersection cone coincides with $\psi(\Nef^{1}(X))$ it is $q$-dual to the dual cone of $\Nef^{1}(X)$.  For (2), by Proposition \ref{hyperkahlerposprod} we have
\begin{align*}
q(\psi(A),\psi(A)) = q(A,A) & = \vol(A)^{2/n} \\
& = \widehat{\vol}(\psi(A))^{2(n-1)/n}.
\end{align*}
For (3), we have
\begin{equation*}
q(B^{n-1},\gamma) = \psi^{-1}(B^{n-1}) \cdot \gamma = \vol(B)^{n-2/n} B \cdot \gamma = 0.
\end{equation*}
For the final statement $q(\gamma,\gamma)<0$, note that
\begin{equation*}
q(\alpha,\alpha) = q(B^{n-1},B^{n-1}) + q(\gamma,\gamma)
\end{equation*}
so it suffices to show that $q(\alpha,\alpha) < q(B^{n-1},B^{n-1})$.  Set $D = \psi^{-1}\alpha$.  The desired inequality is clear if $q(D,D) \leq 0$, so by \cite[Corollary 3.10 and Erratum Proposition 1]{Huy99} it suffices to restrict our attention to the case when $D$ is big.  (Note that the case when $-D$ is big can not occur, since $q(D,A) = A \cdot \alpha > 0$ for an ample divisor class $A$.)  Let $D = P_{\sigma}(D) + N_{\sigma}(D)$ be the $\sigma$-decomposition of $D$.  By \cite[Proposition 4.2]{Bou04} we have $q(N_{\sigma}(D),B) \geq 0$.  Thus
\begin{align*}
\vol(B)^{2(n-1)/n} = q(B^{n-1},B^{n-1}) & = q(\alpha,B^{n-1}) \\
& = \vol(B)^{n-2/n} q(D,B) \geq \vol(B)^{n-2/n}q(P_{\sigma}(D),B).
\end{align*}
Arguing just as in the proof of Proposition \ref{hyperkahlerposprod}, we see that
\begin{equation*}
q(P_{\sigma}(D),B) \geq \vol(P_{\sigma}(D))^{1/n} \vol(B)^{1/n}
\end{equation*}
with equality if and only if $P_{\sigma}(D)$ and $B$ are proportional.  Combining the two previous equations we obtain
\begin{equation*}
\vol(B)^{n-1/n} \geq \vol(P_{\sigma}(D))^{1/n}.
\end{equation*}
and equality is only possible if $B$ and $P_{\sigma}(D)$ are proportional.  Then we calculate:
\begin{align*}
q(\alpha,\alpha) & = q(D,D) \\
& \leq q(P_{\sigma}(D),P_{\sigma}(D)) \textrm{ by \cite[Theorem 4.5]{Bou04}} \\
& = \vol(P_{\sigma}(D))^{2/n} \\
& \leq \vol(B)^{2(n-1)/n} = q(B,B).
\end{align*}
If $P_{\sigma}(D)$ and $B$ are not proportional, we obtain a strict inequality at the last step.  If $P_{\sigma}(D)$ and $B$ are proportional, then $N_{\sigma}(D) > 0$ (since otherwise $D=B$ and $\alpha$ is a complete intersection class).  Then by \cite[Theorem 4.5]{Bou04} we have a strict inequality $q(P_{\sigma}(D),P_{\sigma}(D)) > q(D,D)$ on the second line.  In either case we conclude
$q(\alpha,\alpha) < q(B,B)$ as desired.

\end{proof}


\section{Connections with birational geometry}
\label{applications sec}

We end with a discussion of several connections between positivity of curves and other constructions in birational geometry.  There is a large body of literature relating the positivity of a divisor at a point to its intersections against curves through that point.  One can profitably reinterpret these relationships in terms of the volume of curve classes.  A key result conceptually is:

\begin{prop} \label{multiplicityestimate}
Let $X$ be a smooth projective variety of dimension $n$.  Choose positive integers $\{k_{i} \}_{i=1}^{r}$.  Suppose that $\alpha \in \Mov_{1}(X)$ is represented by a family of irreducible curves such that for any collection of general points $x_{1},x_{2},\ldots,x_{r},y$ of $X$, there is a curve in our family which contains $y$ and contains each $x_{i}$ with multiplicity $\geq k_{i}$.  Then
\begin{equation*}
\widehat{\vol}(\alpha)^{\frac{n-1}{n}} \geq \frac{\sum_{i} k_{i}}{r^{1/n}}.
\end{equation*}
\end{prop}

This is just a rephrasing of well-known results in birational geometry; see for example \cite[V.2.9 Proposition]{k96}.

\begin{proof}
By continuity and rescaling invariance, it suffices to show that if $L$ is a big and nef Cartier divisor class then
\begin{equation*}
\left( \sum_{i=1}^{r} k_{i} \right) \frac{\vol(L)^{1/n}}{r^{1/n}} \leq L \cdot C.
\end{equation*}
A standard argument (see for example \cite[Example 8.22]{lehmann14}) shows that for any $\epsilon > 0$ and any general points $\{ x_{i} \}_{i=1}^{r}$ of $X$ there is a positive integer $m$ and a Cartier divisor $M$ numerically equivalent to $mL$ and such that $\mult_{x_{i}}M \geq mr^{-1/n}\vol(L)^{1/n} - \epsilon$ for every $i$.  By the assumption on the family of curves we may find an irreducible curve $C$ with multiplicity $\geq k_{i}$ at each $x_{i}$ that is not contained $M$.  Then
\begin{equation*}
m(L \cdot C) \geq \sum_{i=1}^{r} k_{i} \mult_{x_{i}}M \geq \left(\sum_{i = 1}^{r} k_{i} \right)  \left(\frac{ m\vol(L)^{1/n}}{r^{1/n}} - \epsilon \right).
\end{equation*}
Divide by $m$ and let $\epsilon$ go to $0$ to conclude.
\end{proof}

\begin{exmple}
The most important special case is when $\alpha$ is the class of a family of irreducible curves such that for any two general points of $X$ there is a curve in our family containing them.  Proposition \ref{multiplicityestimate} then shows that $\widehat{\vol}(\alpha) \geq 1$.
\end{exmple}

\subsection{Seshadri constants}

Let $X$ be a smooth projective variety of dimension $n$ and let $A$ be a big and nef $\mathbb{R}$-Cartier divisor on $X$.  Recall that for points $\{ x_{i} \}_{i=1}^{r}$ on $X$ the Seshadri constant of $A$ along the $\{ x_{i} \}$ is
\begin{equation*}
\varepsilon(x_{1},\ldots,x_{r},A) := \inf_{C \ni x_{i}} \frac{A \cdot C}{\sum_{i} \mult_{x_{i}}C}.
\end{equation*}
where the infimum is taken over all reduced irreducible curves $C$ containing at least one of the points $x_{i}$.  An easy intersection calculation on the blow-up of $X$ at the $r$ points shows that
\begin{equation*}
\varepsilon(x_{1},\ldots,x_{r},A)  \leq \frac{\vol(A)^{1/n}}{r^{1/n}}.
\end{equation*}
When the $r$ points are very general, $r$ is large, and $A$ is sufficiently ample, one ``expects'' the two sides of the inequality to be close.  This heuristic can fail badly, but it is interesting to analyze how close it is to being true.  In particular, the Seshadri constant should only be very small compared to the volume in the presence of a ``Seshadri-exceptional fibration'' (see \cite{ekl95}, \cite{hk03}).
This motivates the following definition:

\begin{defn}
Let $A$ be a big and nef $\mathbb{R}$-Cartier divisor on $X$.  Set $\varepsilon_{r}(A)$ to be the Seshadri constant of $A$ along $r$ points $\mathbf{x} := \{ x_{i} \}$ of $X$.  We define the Seshadri ratio of $A$ to be
\begin{equation*}
sr_{\mathbf{x}}(A) := \frac{r^{1/n} \varepsilon(x_{1},\ldots,x_{r},A)}{\vol(A)^{1/n}}.
\end{equation*}
\end{defn}

Note that the Seshadri ratio is at most $1$, and that low values should only arise in special geometric situations.   The principle established by \cite{ekl95}, \cite{hk03} is that if the Seshadri ratio for $A$ is small, then the curves which approximate the bound in the Seshadri constant can not ``move too much.''

In this section we revisit these known results on Seshadri constants from the perspective of the volume of curves.  In particular we demonstrate how the Zariski decomposition can be used to bound the classes of curves $C$ which give small values in the Seshadri computations above.

\begin{prop} \label{easyseshadriestimate}
Let $X$ be a smooth projective variety of dimension $n$ and let $A$ be a big and nef $\mathbb{R}$-Cartier divisor on $X$.  Fix $\delta > 0$ and fix $r$ points $x_{1},\ldots,x_{r}$.  Suppose that $C$ is a curve containing at least one of the $x_{i}$ and such that
\begin{equation*}
\varepsilon(x_{1},\ldots,x_{r},A) (1+ \delta) > \frac{A \cdot C}{\sum_{i} \mult_{x_{i}}C}.
\end{equation*}
Letting $\alpha$ denote the numerical class of $C$, we have
\begin{equation*}
sr_{\mathbf{x}}(A)  (1+\delta) \geq r^{1/n}\frac{\widehat{\vol}(\alpha)^{n-1/n}}{\sum_{i} \mult_{x_{i}}C}
\end{equation*}
\end{prop}

In fact, this estimate is rather crude; with better control on the relationship between $A$ and $\alpha$, one can do much better.

\begin{proof}
One simply multiplies both sides of the first inequality by $r^{1/n}/\vol(A)^{1/n}$ to deduce that
\begin{align*}
sr_{\mathbf{x}}(A)  (1+\delta) & \geq r^{1/n} \frac{A \cdot C}{\vol(A)^{1/n} \sum_{i} \mult_{x_{i}}C}
\end{align*}
and then uses the obvious inequality $(A \cdot C)/\vol(A)^{1/n} \geq \widehat{\vol}(C)^{n-1/n}$.
\end{proof}

We can then bound the Seshadri ratio of $A$ in terms of the Zariski decomposition of the curve.

\begin{prop}
Let $X$ be a smooth projective variety of dimension $n$ and let $A$ be a big and nef $\mathbb{R}$-Cartier divisor on $X$.  Fix $\delta > 0$ and fix $r$ distinct points $x_{i} \in X$.  Suppose that $C$ is a curve containing at least one of the $x_{i}$ such that the class $\alpha$ of $C$ is big and
\begin{equation*}
\varepsilon(x_{1},\ldots,x_{r},A) (1+ \delta) > \frac{A \cdot C}{\sum_{i} \mult_{x_{i}}C}.
\end{equation*}
Write $\alpha = B^{n-1} + \gamma$ for the Zariski decomposition.  Then $sr_{\mathbf{x}}(A) (1+\delta) > sr_{\mathbf{x}}(B)$.
\end{prop}

\begin{proof}
By Proposition \ref{easyseshadriestimate} it suffices to show that
\begin{equation*}
r^{1/n}\frac{\widehat{\vol}(\alpha)^{n-1/n}}{\sum_{i} \mult_{x_{i}}C} \geq sr_{\mathbf{x}}(B),
\end{equation*}
But this follows from the definition of Seshadri constants along with the fact that $B \cdot C = \widehat{\vol}(C)$.
\end{proof}

These results are of particular interest in the case when the points are very general, when it is easy to deduce the bigness of the class of $C$.

Certain geometric properties of Seshadri constants become very clear from this perspective.  For example, following the notation of \cite{nagata60} we say that a curve $C$ on $X$ is abnormal for a set of $r$ points $\{x_{i}\}$ and a big and nef divisor $A$ if $C$ contains at least one $x_{i}$ and
\begin{equation*}
1 >  \frac{r^{1/n}(A \cdot C)}{\vol(A)^{1/n} \sum_{i} \mult_{x_{i}}C}.
\end{equation*}

\begin{cor}
Let $X$ be a smooth projective variety of dimension $n$ and let $A$ be a big and nef $\mathbb{R}$-Cartier divisor on $X$.  Fix $r$ very general points $x_{1},\ldots,x_{r}$.  Then no abnormal curve goes through a very general point of $X$ aside from the $x_{i}$.
\end{cor}

\begin{proof}
Since the $x_{i}$ are very general, any curve going through at least one more very general point deforms to cover the whole space, so its class is big and nef.  Then combine Proposition \ref{easyseshadriestimate} and Proposition \ref{multiplicityestimate} to deduce that if the Seshadri constant of the $\{x_{i}\}$ is computed by a curve through an additional very general point then $sr_{\mathbf{x}}(A) = 1$.
\end{proof}

\subsection{Rationally connected varieties}
\label{rcexample}

Given a rationally connected variety $X$ of dimension $n$, it is interesting to ask for the possible volumes of curve classes representing rational curves.  In particular, one would like to know if one can find classes whose volumes satisfy a uniform upper bound depending only on the dimension.  There are four natural options:
\begin{enumerate}
\item Consider all classes of rational curves.
\item Consider all classes of chains of rational curves which connect two general points.
\item Consider all classes of irreducible rational curves which connect two general points.
\item Consider all classes of very free rational curves.
\end{enumerate}
Note that each criterion is more special than the previous ones.  We call a class of the second kind an RCC class and a class of the fourth kind a VF class.  Every one of the classes (2), (3), (4) has positive volume; indeed, \cite{8authors} shows that if two general points of $X$ can be connected via a chain of curves of class $\alpha$, then $\alpha$ is a big class.

\

On a Fano variety of Picard rank $1$, the minimal volume of an RCC class is determined by the degree and the minimal degree of an RCC class against the ample generator (or equivalently, the degree, the index, and the length of an RCC class).  The minimum volume is thus related to these well studied invariants.

In higher dimensions, the work of \cite{KMM92} and \cite{campana92} shows that there are constants $C(n)$, $C'(n)$ such that any $n$-dimensional smooth Fano variety carries an RCC class satisfying $-K_{X} \cdot \alpha \leq C(n)$, and a VF class satisfying $-K_{X} \cdot \beta \leq C'(n)$.  We then also obtain explicit bounds on the minimal volume of an RCC or VF class on $X$.  It is interesting to ask what happens for arbitrary rationally connected varieties.

\begin{exmple}
We briefly discuss bounds on the volumes of rational curve classes on smooth surfaces.  Consider first the Hirzebruch surfaces $\mathbb{F}_{e}$.  It is clear that on a Hirzebruch surface a curve class is RCC if and only if it is big, and one easily sees that the minimum volume for an RCC class is $\frac{1}{e}$.  Thus there is no non-trivial universal lower bound for the minimum volume of an RCC class.

In terms of upper bounds, note that if $\pi: Y \to X$ is a birational map and $\alpha$ is an RCC class, then $\pi_{*}\alpha$ is an RCC class as well.  Conversely, given any RCC class $\beta$ on $X$, there is some preimage $\beta'$ on $Y$ which is also an RCC class.  Thus by Proposition \ref{birvolstatement}, we see that any rational surface carries an RCC class of volume no greater than that of an RCC class on a minimal surface.  This shows that any smooth rational surface has an RCC class of volume at most $1$.

On a surface any VF class is necessarily big and nef, so the universal lower bound on the volume is $1$.  In the other direction, consider again the Hirzebruch surface $\mathbb{F}_{e}$.  Any VF class will have the form $aC_{0} + bF$ where $C_{0}$ is the section of negative self-intersection and $F$ is the class of a fiber.  Note that the self intersection is $2ab - a^{2}e$.  For a VF class we clearly must have $a \geq 1$, so that $b \geq ea$ to ensure nefness.  Thus the smallest possible volume of a VF class is $e$, and this is achieved by the class $C_{0} + eF$.  Note that there is no uniform upper bound on the minimum volume of a VF class.
\end{exmple}

As indicated in the previous example, it is most interesting to look for upper bounds on the minimum volume of an RCC class.  Indeed, by taking products with projective spaces, one sees that in any dimension the only uniform lower bound for volumes of RCC classes is $0$.  Furthermore, there is no uniform upper bound for the minimum volume of a VF class.  The crucial distinction is that VF classes are nef, while RCC classes need not be, so that a uniform bound on the volume of a VF class can only be expected for bounded families of varieties.

The following question gives a ``birational'' version of the well-known results of \cite{KMM92}.

\begin{ques}
Let $X$ be a smooth rationally connected variety of dimension $n$.  Is there a bound $d(n)$, depending only on $n$, such that $X$ admits an RCC class of volume at most $d(n)$?
\end{ques}

It is also interesting to ask for optimal bounds on volumes.  The first situation to consider are the ``extremes'' in the examples above.  Note that the lower bound of the volume of a VF class is $1$ by Proposition \ref{multiplicityestimate}, so it is interesting to ask when the minimum is achieved.

\begin{ques}
For which varieties $X$ is the smallest volume of an RCC class equal to $1$?

For which varieties $X$ is the smallest volume of a VF class equal to $1$?
\end{ques}

\section{Appendix}
\subsection{Reverse Khovanskii-Teissier inequalities}
An important step in the analysis of the Morse inequality is the ``reverse" Khovanskii-Teissier inequality for big and nef divisors $A$, $B$, and a movable curve class $\beta$:
\begin{equation*}
n(A \cdot B^{n-1})(B \cdot \beta) \geq B^{n} (A \cdot \beta).
\end{equation*}
We prove a more general statement on ``reverse" Khovanskii-Teissier inequalities in the analytic setting.
Some related work has appeared independently in the recent preprint \cite{popovici15}.

\begin{thrm}
\label{thrm appendix reserve KT}
Let $X$ be a compact K\"ahler manifold of dimension $n$. Let $\omega, \beta, \gamma\in \overline{\mathcal{K}}$ be three nef classes on $X$. Then we have
$$
(\beta^k \cdot \alpha^{n-k})\cdot (\alpha^k\cdot \gamma^{n-k})\geq \frac{k!(n-k)!}{n!}\alpha^n\cdot (\beta^k \cdot\gamma^{n-k}).
$$
\end{thrm}

\begin{proof}
The proof depends on solving Monge-Amp\`{e}re equations and the method of \cite{Pop14}. Without loss of generality, we can assume $\gamma$ is normalised such that $\beta^k\cdot \gamma^{n-k}=1$. Then we need to show
\begin{align}
\label{eq need to prove}
(\beta^k \cdot \alpha^{n-k})\cdot (\alpha^k\cdot \gamma^{n-k})\geq \frac{k!(n-k)!}{n!}\alpha^n.
\end{align}
We first assume $\alpha,\beta,\gamma$ are all K\"ahler classes. We will use the same symbols to denote the K\"ahler metrics in corresponding K\"ahler classes. By the Calabi-Yau theorem \cite{Yau78}, we can solve the following Monge-Amp\`{e}re equation:
\begin{align}
\label{eq MA}
(\alpha+i\partial\bar \partial \psi)^n= \left(\int \alpha^n \right) \beta^k\wedge \gamma^{n-k}.
\end{align}
Denote by $\alpha_\psi$ the K\"ahler metric $\alpha+i\partial\bar \partial \psi$. Then we have
\begin{align*}
(\beta^k \cdot \alpha^{n-k})\cdot (\alpha^k\cdot \gamma^{n-k})&=\int \beta^k \wedge \alpha_\psi^{n-k}\cdot \int \alpha_\psi^k \wedge \gamma^{n-k}\\
&= \int \frac{\beta^k \wedge \alpha_\psi^{n-k}}{\alpha_\psi ^n}\alpha_\psi ^n\cdot \int \frac{\alpha_\psi^k \wedge \gamma^{n-k}}{\alpha_\psi ^n}\alpha_\psi ^n\\
&\geq \left( \int \left(\frac{\beta^k \wedge \alpha_\psi^{n-k}}{\alpha_\psi ^n} \cdot \frac{\alpha_\psi^k \wedge \gamma^{n-k}}{\alpha_\psi ^n} \right)^{1/2}\alpha_\psi ^n\right)^2.
\end{align*}

The last line follows because of the Cauchy-Schwarz inequality. We claim that the following pointwise inequality holds:
\begin{align*}
\frac{\beta^k \wedge \alpha_\psi^{n-k}}{\alpha_\psi ^n} \cdot \frac{\alpha_\psi^k \wedge \gamma^{n-k}}{\beta^k \wedge \gamma^{n-k}} \geq \frac{k!(n-k)!}{n!}.
\end{align*}
Then by (\ref{eq MA}) it is clear the above pointwise inequality implies the desired inequality (\ref{eq need to prove}). For any fixed point $p\in X$, we can choose some coordinates such that at the point $p$: $$\alpha_\psi= i\sum_{j=1}^n dz^j \wedge d\bar z ^j,
\quad \beta = i\sum_{j=1}^n \mu_j dz^j \wedge d\bar z ^j,
$$
and
 $$\gamma^{n-k}=i^{n-k} \sum_{|I|=|J|=n-k} \Gamma_{IJ} dz_I \wedge d\bar z_{J}.$$
Denote by $\mu_J$ the product $\mu_{j_1}...\mu_{j_k}$ with index $J=(j_1<...<j_k)$ and denote by $J^c$ the complement index of $J$. Then it is easy to see at the point $p$ we have
$$
\frac{\beta^k \wedge \alpha_\psi^{n-k}}{\alpha_\psi ^n} \cdot \frac{\alpha_\psi^k \wedge \gamma^{n-k}}{\beta^k \wedge \gamma^{n-k}}= \frac{k!(n-k)!}{n!} \frac{(\sum_{J}\mu_J)(\sum_K \Gamma_{KK})}{\sum_{J}\mu_J \Gamma_{J^c J^c}}\geq \frac{k!(n-k)!}{n!}.
$$
This finishes the proof of the case when $\alpha,\beta,\gamma$ are all K\"ahler classes. If they are just nef classes, by taking limits, then we get the desired inequality.
\end{proof}

\begin{rmk}
\label{rmk appendix movable}
By \cite[Section 2.1.1]{xiao15}, for $k=1$ we can always replace $\gamma^{n-1}$ in Theorem \ref{thrm appendix reserve KT} by an arbitrary movable class.
\end{rmk}

\begin{rmk}
It would be interesting to find an algebraic approach to Theorem \ref{thrm appendix reserve KT}, thus generalizing it to projective varieties defined over arbitrary fields.
\end{rmk}

\subsection{Towards the transcendental holomorphic Morse inequality}

Recall that the (weak) transcendental holomorphic Morse inequality over compact K\"ahler manifolds conjectured by Demailly is stated as follows:
\begin{itemize}
\item Let $X$ be a compact K\"ahler manifold of dimension $n$, and let $\alpha, \beta\in \overline{\mathcal{K}}$ be two nef classes. Then we have $\vol(\alpha-\beta)\geq \alpha^n -n \alpha^{n-1}\cdot \beta$. In particular, if $\alpha^n -n \alpha^{n-1}\cdot \beta>0$ then there exists a K\"ahler current in the class $\alpha-\beta$.
\end{itemize}
Indeed, the last statement has been proved in the recent work \cite{Xia13, Pop14}. The missing part is how to bound the volume $\vol(\alpha-\beta)$ by $\alpha^n -n \alpha^{n-1}\cdot \beta$.

By \cite[Theorem 2.1 and Remark 2.3]{xiao15} the volume for transcendental pseudo-effective $(1,1)$-classes is conjectured to be characterized as following:
\begin{align}
\label{eq vol conj}
\vol(\alpha)=\inf_{\gamma\in \mathcal{M}, \mathfrak{M}(\gamma)=1}(\alpha\cdot \gamma)^n
\end{align}
For the definition of $\mathfrak{M}$ in the K\"ahler setting, see \cite[Definition 2.2]{xiao15}. If we denote the right hand side of (\ref{eq vol conj}) by $\overline{\vol}(\alpha)$, then we can prove the following:

\begin{thrm}
\label{thm tran morse}
Let $X$ be a compact K\"ahler manifold of dimension $n$, and let $\alpha, \beta\in \overline{\mathcal{K}}$ be two nef classes. Then we have $$\overline{\vol}(\alpha-\beta)^{1/n}\vol(\alpha)^{n-1/n}\geq \alpha^n -n \alpha^{n-1}\cdot \beta.$$
\end{thrm}

\begin{proof}
We only need to consider the case when $\alpha^n -n \alpha^{n-1}\cdot \beta>0$. And \cite{Pop14} implies the class $\alpha-\beta$ is big. By the definition of $\overline{\vol}$, we have
$$
\overline{\vol}(\alpha-\beta)^{1/n}=\inf_{\gamma\in \mathcal{M}, \mathfrak{M}(\gamma)=1}(\alpha-\beta)\cdot \gamma.
$$
So we need to estimate $(\alpha-\beta)\cdot \gamma$ with $\mathfrak{M}(\gamma)=1$:
\begin{align*}
(\alpha-\beta)\cdot \gamma&=\alpha\cdot \gamma-\beta\cdot \gamma\\
&\geq \alpha\cdot \gamma-\frac{n(\alpha^{n-1}\cdot \beta)\cdot(\alpha\cdot\gamma)}{\alpha^n}\\
&=\frac{\alpha\cdot\gamma}{\alpha^n}(\alpha^n-n\alpha^{n-1}\cdot \beta)\\
&\geq \vol(\alpha)^{1-n/n}(\alpha^n-n\alpha^{n-1}\cdot \beta),
\end{align*}
where the second line follows from Theorem \ref{thrm appendix reserve KT} and Remark \ref{rmk appendix movable}, and the last line follows the definition of $\mathfrak{M}$ and $\mathfrak{M}(\gamma)=1$.

By the arbitrariness of $\gamma$ we get
$$\overline{\vol}(\alpha-\beta)^{1/n}\vol(\alpha)^{n-1/n}\geq \alpha^n -n \alpha^{n-1}\cdot \beta. $$
\end{proof}

\begin{rmk}
Without using the conjectured equality (\ref{eq vol conj}), it is observed independently by \cite{tosatti2015current} and \cite{popovici15} that one can replace $\overline{\vol}$ by the volume function $\vol$ in Theorem \ref{thm tran morse}.
\end{rmk}

\subsection{Non-convexity of the complete intersection cone}

We give an example explicitly verifying the non-convexity of $\CI_{1}(X)$.  

\begin{exmple}
\cite{fs09} gives an example of a smooth toric threefold $X$ such that every nef divisor is big.  We show that for this toric variety $\CI_{1}(X)$ is not convex.

 Let $X$ be the toric variety defined by a fan in $N = \mathbb{Z}^{3}$ on the rays
\begin{align*}
v_{1} & = (1,0,0)  \qquad \qquad v_{2} = (0,1,0)  & v_{3} & = (0,0,1) \qquad \qquad v_{4} = (-1, -1,-1) \\
v_{5} & = (1,-1,-2) \qquad \, \, \, v_{6} = (1,0,-1) & v_{7} & = (0,-1,-2) \qquad \, \, \, v_{8} = (0,0,-1)
\end{align*}
with maximal cones
\begin{align*}
\langle v_{1}, v_{2}, v_{3} \rangle, \, \langle v_{1}, v_{2}, v_{6} \rangle, \, \langle v_{1}, v_{3}, v_{4} \rangle, \, \langle v_{1}, v_{4}, v_{5} \rangle, \\
\langle v_{1}, v_{5}, v_{6} \rangle, \, \langle v_{2}, v_{3}, v_{4} \rangle, \, \langle v_{2}, v_{4}, v_{8} \rangle, \, \langle v_{2}, v_{5}, v_{6} \rangle, \\
\langle v_{2}, v_{5}, v_{8} \rangle, \, \langle v_{4}, v_{5}, v_{7} \rangle, \, \langle v_{4}, v_{7}, v_{8} \rangle, \, \langle v_{5}, v_{7}, v_{8} \rangle.
\end{align*}

Since $X$ is the blow-up of $\mathbb{P}^{3}$ along 4 rays, it has Picard rank $5$.  Let $D_{i}$ be the divisor corresponding to the ray $v_{i}$ and $C_{ij}$ denote the curve corresponding to the face generated by $v_{i}$ and $v_{j}$.  Standard toric computations show that the pseudo-effective cone of divisors is simplicial and is generated by $D_{1},D_{5},D_{6},D_{7},D_{8}$.  The pseudo-effective cone of curves is also simplicial and is generated by $C_{14}, C_{16}, C_{25}, C_{47}, C_{48}$.  From now on we will write divisor or curve classes as vectors in these (ordered) bases.

The intersection matrix is:
\begin{equation*}
\begin{array}{c|c|c|c|c|c}
& D_{1} & D_{5} & D_{6} & D_{7} & D_{8} \\ \hline
C_{14} & -2 & 1 & 0 & 0 & 0 \\ \hline
C_{16} & 1 & 1 & -2 & 0 & 0 \\ \hline
C_{25} & 0 & -1 & 1 & 0 & 1 \\ \hline
C_{47} & 0 & 1 & 0 & -2 & 1 \\ \hline
C_{48} & 0 & 0 & 0 & 1 & -2
\end{array}
\end{equation*}
The nef cone of divisors is dual to the pseudo-effective cone of curves.  Thus it is simplicial and has generators $A_{1},\ldots,A_{5}$ determined by the columns of the inverse of the matrix above:
\begin{align*}
A_{1} & = (1,3,2,2,1) \\
A_{2} & = (3,6,4,4,2) \\
A_{3} & = (6,12,9,8,4) \\
A_{4} & = (2,4,3,2,1) \\
A_{5} & = (4,8,6,5,2)
\end{align*}


A computation shows that for real numbers $x_{1},\ldots,x_{5}$,
\begin{align*}
\left( \sum_{i=1}^{5} x_{i}A_{i} \right)^{2} = & (1,3,6,2,4) (x_{1}^{2}+6x_{1}x_{2}+12x_{1}x_{3}+4x_{1}x_{4} + 8x_{1}x_{5}) + \\
& (9,22,45,15,30) x_{2}^{2} + \\
& (12,30,60,20,40) (x_{2}x_{4}+2x_{2}x_{5}+3x_{2}x_{3}+3x_{3}^{2}+2x_{3}x_{4}+4x_{3}x_{5}) + \\
& (4,10,20,6,13) x_{4}^{2} + \\
& (16,40,80,26,52) (x_{4}x_{5} + x_{5}^{2}).
\end{align*}

Note that the five vectors above form a basis of $N_{1}(X)$ and each one is proportional to one of the $A_{i}^{2}$.

It is clear from this explicit description that the cone is not convex.  For example, the vector
\begin{equation*}
v = (9,22,45,15,30) + (4,10,20,6,13)
\end{equation*}
can not be approximated by curves of the form $H^{2}$ for an ample divisor $H$.  Indeed, if we have a sequence of ample divisors $H_{j} = \sum x_{i,j}A_{i}$ with $x_{i,j} > 0$ such that $H_{j}^{2}$ converges to $v$, then
\begin{equation*}
\lim_{j \to \infty} x_{2,j} = 1 \qquad \qquad \textrm{and} \qquad \qquad \lim_{j \to \infty} x_{4,j} = 1.
\end{equation*}
But then the limit of the coefficient of $(12,30,60,20,40)$ is at least $1$, a contradiction.  Exactly the same argument shows that the closure of the set of all products of two (possibly different) ample divisors is not convex.
\end{exmple}

\bibliography{cnvxzar}
\bibliographystyle{amsalpha}

\noindent
\textsc{Brian Lehmann}\\
\textsc{Department of Mathematics, Boston College,
Chestnut Hill, MA 02467, USA}\\
\verb"Email: lehmannb@bc.edu"\\

\noindent
\textsc{Jian Xiao}\\
\textsc{Institute of Mathematics, Fudan University, 200433 Shanghai, China}\\

\noindent
\textsc{Current address:}\\
\textsc{Institut Fourier, Universit\'{e} Joseph Fourier, 38402 Saint-Martin d'H\`{e}res, France}\\
\verb"Email: jian.xiao@ujf-grenoble.fr"\\

\end{document}